\documentclass[preprint,12pt]{elsarticle}

\usepackage{amssymb}

\def\N{{\Bbb N}}
\def\Z{{\Bbb Z}}
\def\R{{\Bbb R}}
\def\T{{\Bbb T}}
\def\C{{\Bbb C}}

\def\bm{\begin{matrix}}
\def\em{\end{matrix}}

\newcommand{\CH}{{\mathcal H}}
\newcommand{\CI}{{\mathcal I}}
\newcommand{\CL}{{\mathcal L}}

\newcommand{\CO}{{\mathcal O}}

\newtheorem{Theorem}{Theorem}
\newtheorem{Lemma}{Lemma}[section]
\newtheorem{Proposition}{Proposition}
\newtheorem{Corollary}{Corollary}
\newtheorem{Remark}{Remark}[section]

\newcommand{\proof}{\par\medbreak\it Proof: \rm}
\newcommand{\la}{\langle }
\newcommand{\ra}{\rangle }

\journal{J. Math. Pures Appl.}

\begin{document}

\begin{frontmatter}



\title{1-d Quantum Harmonic Oscillator with Time Quasi-periodic Quadratic Perturbation:
Reducibility and Growth of Sobolev Norms}


\author{Zhenguo Liang $^{\rm a},$\footnote{Z. Liang was partially supported by NSFC grant (11371097, 11571249) and Natural Science Foundation of Shanghai (19ZR1402400). }, Zhiyan Zhao $^{{\rm b},}$\footnote{Corresponding author. The work of Z. Zhao was partially supported by the French government through the National Research Agency (ANR) grant ANR-15-CE40-0001-03 for the project BeKAM and through the UCA JEDI Investments in the Future project managed by ANR with the reference number ANR-15-IDEX-01},  Qi Zhou $^{{\rm c},}$\footnote{ Q. Zhou was partially supported by support by NSFC grant (11671192, 11771077), the Science Fund for Distinguished Young Scholars of Tianjin (No. 19JCJQJC61300) and Nankai Zhide Foundation.\\
E-mail addresses: zgliang@fudan.edu.cn(Z. Liang), zhiyan.zhao@univ-cotedazur.fr(Z. Zhao), qizhou@nankai.edu.cn(Q. Zhou)}}

\address{$^{\rm a}$ School of Mathematical Sciences and Key Lab of Mathematics for Nonlinear Science, Fudan University, Shanghai 200433, China\\[2mm]
$^{\rm b}$ Universit\'e C\^ote d'Azur, CNRS, Laboratoire J. A. Dieudonn\'{e}, 06108 Nice, France\\[2mm]
$^{\rm c}$ Chern Institute of Mathematics and LPMC, Nankai University, Tianjin 300071, China}

\begin{abstract}

For a family of 1-d quantum harmonic oscillators with a perturbation which is $C^2$ parametrized by $E\in{\CI}\subset\R$ and quadratic on $x$ and $-{\rm i}\partial_x$ with coefficients quasi-periodically depending on time $t$, we show the reducibility (i.e., conjugation to time-independent) for a.e. $E$. As an application of reducibility, we describe the behaviors of solutions in Sobolev space:
\begin{itemize}
  \item Boundedness w.r.t. $t$ is always true for ``most" $E\in{\CI}$.
  \item For ``generic" time-dependent perturbation, polynomial growth and exponential growth to infinity w.r.t. $t$ occur for $E$ in a ``small" part of ${\CI}$.
\end{itemize}
Concrete examples are given for which the growths of Sobolev norm do occur.

\smallskip

\noindent{\bf R\'{e}sum\'{e}}

Pour une famille des oscillateurs harmoniques quantiques unidimensionnels avec une perturbation qui est parametr\'ee par $E\in{\CI}\subset\R$ d'une mani\`ere $C^2$ et qui est quadratique sur $x$ et $-{\rm i}\partial_x$ avec des coefficients qui d\'ependent du temps $t$ d'une mani\`ere quasi-periodique, on montre la r\'eductibilit\'e (c'est-\`a-dire la conjugaison \`a l'ind\'ependant du temps) pour presque tout $E$. Comme une application de la r\'eductibilit\'e, on d\'ecrit les comportements des solutions dans l'espace de Sobolev:
\begin{itemize}
  \item La bornitude par rapport \`a $t$ est toujours vraie pour la $\ll$ plupart $\gg$ de $E\in{\CI}$.
  \item Pour la perturbation $\ll$ g\'en\'erique $\gg$ qui d\'epend du temps, la croissance polynomiale et la croissance exponentielle \`a l'infini par rapport \`a $t$ ont lieu pour $E$ dans une $\ll$ petite $\gg$ partie de ${\CI}$.
\end{itemize}
Des exemples concrets sont donn\'es pour lesquels les croissances de la norme de Sobolev vraiment ont lieu.

\end{abstract}

\end{frontmatter}


\section{Introduction and main results}

Consider the one-dimensional Schr\"odinger equation
\begin{equation}\label{eq_Schrodinger}
 {\rm i}\partial_t u=\frac{\nu(E)}{2}H_0u + W(E,\omega t, x, -{\rm i} \partial_x)u,\qquad x\in\R ,
\end{equation}
where, we assume that
\begin{itemize}
  \item the frequencies $\omega\in \R^d$, $d\geq 1$, satisfy the {\it Diophantine} condition (denoted by $\omega\in {\rm DC}_d(\gamma,\tau)$ for $\gamma>0$, $\tau>d-1$):
$$\inf_{j\in\Z}|\la n,\omega\ra-j|>\frac{\gamma}{|n|^\tau},\qquad \forall \ n\in\Z^d\setminus\{0\},$$
  \item the parameter $E\in {\CI}$, an interval $\subset \R$, and $\nu\in C^2({\CI},\R)$ satisfies
  $$|\nu'(E)|\geq l_1,\quad |\nu''(E)|\leq l_2,\qquad \forall \  E\in{\CI},$$
for some $l_1,l_2>0$,
  \item $H_0$ is the {\it one-dimensional quantum harmonic oscillator}, i.e.
$$(H_0u)(x):=-(\partial_{x}^2 u)(x)+x^2\cdot u(x),\qquad \forall \ u\in L^2(\R),$$
  \item $W(E,\theta, x,\xi)$ is a quadratic form of $(x,\xi)$:
$$W(E, \theta, x,\xi)=\frac12 \big(a(E,\theta) x^2+2b(E,\theta) x\cdot\xi+c(E,\theta) \xi^2\big),$$
with $a,b,c:{\CI}\times \T^d\to \R$, all of which are $C^2$ w.r.t. $E\in{\CI}$ and $C^\omega$ w.r.t. $\theta\in\T^d:=(\R/\Z)^d$, and for every $E\in{\CI}$, for $m=0,1,2$,
$|\partial^m_E a(E,\cdot)|_r:=\sup_{|\Im z|<r}|\partial^m_E a(E,z)|$, $|\partial^m_E b(E,\cdot)|_r$, $|\partial^m_E c(E,\cdot)|_r$ are small enough.
\end{itemize}
We will prove that, for almost every $E$ in the interval ${\CI}$, Eq. (\ref{eq_Schrodinger}) is reducible, i.e., via a unitary transformation, Eq. (\ref{eq_Schrodinger}) is conjugated to an equation which is independent of time (while the transformation depends on time in an analytic quasi-periodic way).
According to the reducibility, we deduce the behavior of Sobolev norms for the solutions to Eq. (\ref{eq_Schrodinger}).

\subsection{Reducibility for harmonic oscillators}

Our main result is the following:

\begin{Theorem}\label{thm_Schro}
There exists $\varepsilon_*=\varepsilon_*(\gamma,\tau,r,d,l_1,l_2)>0$ such that if
$$\max_{m=0,1,2}\left\{\left|\partial_E^m a\right|_r, \ \left|\partial_E^m b\right|_r, \ \left|\partial_E^m c\right|_r\right\}=:\varepsilon_0\leq \varepsilon_*,\qquad \forall \ E\in{\CI},$$
then for a.e. $E\in{\CI}$, Eq. (\ref{eq_Schrodinger}) is reducible, i.e., there exists a time quasi-periodic transformation $U(\omega t)$, unitary in $L^2$ and analytically depending on $t$, such that
Eq. (\ref{eq_Schrodinger}) is conjugated to
${\rm i}\partial_t v= G v$ by the transformation $u=U(\omega t) \, v$, with $G$ a linear operator independent of $t$.

More precisely,
there exists a subset
$${\CO}_{\varepsilon_0}=\bigcup_{j\in\N}\Lambda_j\subset \overline{\CI}$$ with $\Lambda_j$'s being closed intervals \footnote{In this paper, the ``closed interval" is interpreted in a more general sense, i.e., it can be degenerated to a point instead of a positive-measure subset of $\R$.}
and ${\rm Leb}({\CO}_{\varepsilon_0})<\varepsilon_0^{\frac{1}{40}}$,
such that the following holds.
\begin{enumerate}
\item For a.e. $E\in{\CI}\setminus{\CO}_{\varepsilon_0}$, $G$ is unitary equivalent to $\varrho H_0$ for some $\varrho=\varrho_E\geq 0$;
\item If ${\rm Leb}(\Lambda_j)>0$, then
\begin{itemize}
  \item for $E\in{\rm int}\Lambda_j$, $G$ is unitary equivalent to $-\frac{\lambda {\rm i}}{2}(x\cdot \partial_x+ \partial_x \cdot x)$ for some $\lambda=\lambda_E> 0$;
  \item for $E\in\partial\Lambda_j\setminus\partial{\CI}$, $G$ is unitary equivalent to $-\frac{\kappa}2 x^2$ for some $\kappa=\kappa_E\in\R\setminus\{0\}$.
\end{itemize}
If ${\rm Leb}(\Lambda_j)=0$, then $G=0$ for $E\in\Lambda_j$.
\end{enumerate}
\end{Theorem}

Before giving its application on the growth of Sobolev norm, let us first make a review on previous works about the reducibility on harmonic oscillators as well as the relative KAM theory.

For 1-d harmonic oscillators with time periodic smooth perturbations, Combescure \cite{Com87} firstly showed the pure point nature of Floquet operator (see also \cite{DLSV2002, EV83, Kuk1993}).
For 1-d harmonic oscillators with time quasi-periodic bounded perturbations, we can refer to \cite{GreTho11, Wang08, WLiang17} for the reducibility and the pure point spectrum of Floquet operator.
For 1-d harmonic oscillators with unbounded time quasi-periodic perturbations, similar results can be found in \cite{Bam2018, Bam2017, BM2018, Liangluo19}.
In investigating the reducibility problems, KAM theory for 1-d PDEs has been well developed by Bambusi-Graffi \cite{BG2001} and Liu-Yuan \cite{LY2010} in order to deal with unbounded perturbations.

Reducibility for PDEs in higher-dimensional case was initiated by Eliasson-Kuksin \cite{EK2009}, based on their KAM theory \cite{EK2010}.
We refer to \cite{GrePat16} and \cite{LiangWang19} for higher-dimensional harmonic oscillator with bounded potential.
We mention that some higher-dimensional results with unbounded perturbations have been recently obtained
\cite{BLM18, FGMP19, FG19, FGN19, Mon19}.
However, a general KAM theorem for higher-dimensional PDEs with unbounded perturbations is far from success.

Recently, Bambusi-Gr\'ebert-Maspero-Robert \cite{BGMR2018} built a reducibility result for the harmonic oscillators on $\R^n$, ,$n\geq 1$, in which the perturbation is a polynomial of degree at most two in $x$ and $-{\rm i}\partial_x$ with coefficients quasi-periodically depending on time.
The proof in \cite{BGMR2018} exploits the fact that for polynomial Hamiltonians of degree at most $2$, there is an exact correspondence between classical and quantum mechanics, so that the result can be proved by exact quantization of the classical KAM theory which ensures reducibility of the classical Hamiltonian system.
The exact correspondence between classical and quantum dynamics of quadratic Hamiltonians was already exploited in the paper \cite{HLS86} to prove stability and instability results for one degree of freedom time periodic quadratic Hamiltonians.
To prove our main result, we use the same strategy as \cite{BGMR2018} and the reducibility result for the classical Hamiltonian by Eliasson \cite{Eli1992}.

\subsection{Growth of Sobolev norms}

Besides reducibility, the construction of unbounded solutions in Sobolev space for Schr\"odinger equations  attracts even more attentions.

As an application of Theorem \ref{thm_Schro}, we can study the long time behaviour of its solution $u(t)$ to Eq. (\ref{eq_Schrodinger}) in Sobolev space.
For $s\geq 0$, we define Sobolev space
$${\CH}^s:=\left\{\psi\in L^2(\R):H_0^{\frac{s}2}\psi \in L^2(\R)\right\}$$
and Sobolev norm
$\|\psi\|_{s}:=\|H_0^{\frac{s}2} \psi\|_{L^2(\R)}$.
It is well known that, for $s\in \N$, the above definition of norm is equivalent to
$$
\sum\limits_{\alpha+\beta\leq s\atop{\alpha,\beta\in\N} }\|x^{\alpha}\cdot\partial^{\beta} \psi\|_{L^2(\R)}.
$$
\begin{Remark}\label{remark_norm_equiv}
 In view of Remark 2.2 of \cite{BM2018}, we get that, for a given $\psi\in {\CH}^s$,
\begin{equation}\label{norm_equiv}
\|\psi\|_{s}\simeq \|\psi\|_{H^s}+ \|x^{s} \psi\|_{L^2},
\end{equation}
replacing $K_0=H_0$ in that remark by $K_0=H_0^{\frac12}$, where $H^s$ means the standard Sobolev space and $\|\cdot\|_{H^s}$ is the corresponding norm. Hence, to calculate the norm $\|\psi\|_s$, $s\geq 0$, it is sufficient to focus on $\|x^{s} \psi\|_{L^2}$ for $s\geq0$ and $\|\psi^{(s)}\|_{L^2}$ for $s\in\N$.
\end{Remark}

For different types of reduced systems, Sobolev norm of solution exhibits different behaviors.

\begin{Theorem}\label{thm_Schro_sobolev} Under the assumption of Theorem \ref{thm_Schro}, for any $s\geq 0$, and any non-vanishing initial condition $u(0)\in {\CH}^s$, the following holds true for the solution $u(t)$ to Eq. (\ref{eq_Schrodinger}) for $t\geq0$.
\begin{enumerate}
\item For a.e. $E\in{\CI}\setminus{\CO}_{\varepsilon_0}$,
$
c \leq \|u(t)\|_{s}\leq C
$.
\item If ${\rm Leb}(\Lambda_j)>0$, then
\begin{itemize}
  \item for $E\in{\rm int}\Lambda_j$,
$c  e^{\lambda st}  \leq \|u(t)\|_{s} \leq C  e^{\lambda st}$,
  \item for $E\in\partial \Lambda_j\setminus \partial{\CI}$,
$
c |\kappa|^s t^s \leq \|u(t)\|_{s}\leq C |\kappa|^s (1+ t^2)^{\frac{s}2}
$.
\end{itemize}
If ${\rm Leb}(\Lambda_j)=0$, then for $E\in \Lambda_j$, $c   \leq \|u(t)\|_{s} \leq C $.
\end{enumerate}
Here $\lambda=\lambda_E$ and $\kappa=\kappa_E$ are the same with Theorem \ref{thm_Schro} and $c, \, C>0$ are two constants depending on $s$, $E$ and $u(0)$.
\end{Theorem}

\begin{Remark} For a.e. $E\in{\CI}\setminus{\CO}_{\varepsilon_0}$, we have $0<c< C<\infty$. However, we could not expect the uniformity of these two constants and the ratio between them w.r.t. $E$. Indeed, the two constants are indeed influenced by the quasi-periodic unitary transformation $U(\omega t)$ and the (constant) unitary equivalence of $G$ obtained in Theorem \ref{thm_Schro}. Since both unitary transformations may not be close to identity,
$c$ and $C$ are usually not close to each other. Even though in the ``simplest" case, i.e., no resonance occurs in the classical and quantum KAM iteration (see Proof of Proposition \ref{prop_eliasson} (1) for details) and hence $U(\omega t)$ is close to identity, the ratio between $c$ and $C$ is not always close to $1$ because of the nondeterminacy of $G$ (see Remark 1.2 of \cite{BGMR2018}). Moreover, different $E$ would have different  scales of resonances, then $c$, $C$ would lose uniform control w.r.t $E$.
\end{Remark}

Let us review the progress on constructing unbounded solutions of time-dependent Schr\"odinger equations.
For linear Schr\"odinger equation on $\T$ with time quasi-periodic perturbation, by exploiting resonance effects, Bourgain \cite{Bou99} built logarithmic upper bound for Sobolev norm of solutions and constructed examples of solution which exhibits logarithmic growth with $t$.
Later, for 1-d harmonic oscillator with certain time periodic order zero perturbation, Delort \cite{Del2014} constructed some solution with its Sobolev norm polynomially growing with $t$.
By exploiting the idea in \cite{GY00}, Maspero \cite{Mas2018} gave a simplified proof for the result of Delort \cite{Del2014}.
In \cite{BGMR2018}, the authors also considered the higher-dimensional harmonic oscillator with time quasi-periodic perturbation which is linear in $x$ and $-{\rm i}\partial_x$. Under the Diophantine condition on frequencies, the time-dependent equation can be reduced to a special ``normal form" independent of time (see Theorem 3.3 of \cite{BGMR2018}), which implies the polynomial growth of Sobolev norm. In particular,
a concrete example of such polynomial growth is given for 1-d harmonic oscillator with time periodic perturbation (see Corollary A.2 of \cite{BGMR2018}). Recently, for 2-d harmonic oscillator with perturbation which is decaying in $t$, Faou-Rapha\"el \cite{FR20} constructed a solution whose ${\CH}^1-$norm presents logarithmic growth with $t$. For 2-d harmonic oscillator with perturbation being the projection onto Bargmann-Fock space, Thomann \cite{Thomann20} constructed explicitly a travelling wave whose Sobolev norm presents polynomial growth with $t$, based on the study in \cite{ST20} for linear Lowest Landau Level equations (LLL) with a time-dependent potential.
There are also many literatures, e.g., \cite{BGMR2019, MR2017}, which are relative to the upper growth bound of the solution in Sobolev space.

\smallskip

From the above mentioned literatures, we can see that the growth of Sobolev norm of solution is closely related to the resonance phenomenon.
However, it was not clear to us how growth and boundedness coexist with each other.
Following \cite{Eli1992}, for 1-d harmonic oscillator with quadratic perturbation (\ref{eq_Schrodinger}), we introduce in this paper the parameter set $\bigcup_{j\in\N}\Lambda_j$, in which the solutions have exponential growth with $t$, while on the boundaries of this set the solutions present polynomial growth with $t$. This subset gives a geometric description on the transition between boundedness and two types of growth.

\subsection{Examples with ${\rm Leb}({\CO}_{\varepsilon_0})>0$}

In the following, we will present several concrete examples of time quasi-periodic quadratic perturbations for which the set $\bigcup_{j\in\N}\Lambda_j$ is of positive measure.

In view of Theorem \ref{thm_Schro} and \ref{thm_Schro_sobolev}, the growth of Sobolev norm can be obtained via the reducibility if ${\rm Leb}({\CO}_{\varepsilon_0})>0$.
We need to point that the time-dependent quadratic perturbation $W(E,\omega t, x, -{\rm i} \partial_x)$ with ${\rm Leb}({\CO}_{\varepsilon_0})>0$ exists universally. In other words, it is a quite ``extreme" case that
$${\rm Leb}(\Lambda_j)=0, \qquad \forall \ j\in\N.$$
We have the following concrete examples.

\

For ${\CI}=\R$, $\nu(E)=E$, the equation
\begin{equation}\label{example_1}
 {\rm i}\partial_t u=\frac{E}{2}H_0u + \left(\frac{a(\omega t)}{2} x^2-\frac{b(\omega t)}2\left(x\cdot{\rm i}\partial_x+{\rm i}\partial_x\cdot x\right)-\frac{c(\omega t)}2 \partial_x^2 \right) u,
\end{equation}
satisfies the assumptions of Theorem \ref{thm_Schro} if $a,b,c\in C^\omega(\T^d,\R)$ are small enough. Hence, for Eq. (\ref{example_1}), the reducibility and the behaviors of ${\CH}^s$ norm of solutions described in Theorem \ref{thm_Schro_sobolev} can be obtained.
\begin{Theorem}\label{thm_example_1}
For generic $a,b,c\in C^\omega(\T^d,\R)$ with $|a|_r, |b|_r, |c|_r$ small enough (depending on $r,\gamma,\tau,d$), the conclusions of Theorem \ref{thm_Schro} and \ref{thm_Schro_sobolev} hold for Eq. (\ref{example_1}) for ${\CI}=\R$ with ${\rm Leb}({\CO}_{\varepsilon_0})>0$.
\end{Theorem}

\

For $\nu(E)=\sqrt{E}$, consider the equation
\begin{equation}\label{eq_Schrodinger-example}
 {\rm i}\partial_t u=\frac{\sqrt{E}}{2} H_0 u -\frac{q(\omega t)}{2\sqrt{E}}\left(x^2-x\cdot{\rm i}\partial_x-{\rm i}\partial_x\cdot x-\partial^2_x \right)u.
\end{equation}
with $q\in C_r^\omega(\T^d,\R)$. The equation is important, since as we will show later, it is closely related to quasi-periodic Schr\"odinger operator.
\begin{Theorem}\label{thm_example_schro}
For generic $q\in C^\omega(\T^d,\R)$, the conclusions of Theorem \ref{thm_Schro} and \ref{thm_Schro_sobolev} hold for Eq. (\ref{eq_Schrodinger-example}) for ${\CI}=[E_0,E_1]$ with ${\rm Leb}(\Lambda_j)>0$ for infinitely many $j$'s, where $E_0>0$ is large enough (depending on $|q|_r$) and $E_1<\infty$.
\end{Theorem}

\

Theorem \ref{thm_example_1} gives the example that ${\rm Leb}(\Lambda_j)>0$ for at least one $j$, while Theorem \ref{thm_example_schro} gives the  example that  ${\rm Leb}(\Lambda_j)>0$ for infinitely many $j$'s. Indeed, if the dimension of the frequency $d=2$, we could even give examples for which ${\rm Leb}(\Lambda_j)>0$ for every $j$'s.  To construct such an example, we consider
 \begin{equation}\label{eq_AMO}
  {\rm i}\partial_t u=\frac{\nu(E)}{2} H_0 u+ \left(\frac{a(E,\omega t)}{2} x^2-\frac{b(E,\omega t)}2\left(x\cdot{\rm i}\partial_x+{\rm i}\partial_x\cdot x\right)-\frac{c(E,\omega t)}2 \partial_x^2 \right) u.
 \end{equation}
where $\nu(E)=\cos^{-1}(-\frac{E}{2})$,  ${\CI}\subset[-2+\delta,2-\delta]$ with $\delta$ a small numerical constant (e.g., $\delta=10^{-6}$). Then our result is the following:

 \begin{Theorem}\label{thm_AMO} There exist a sub-interval ${\CI}\subset[-2+\delta,2-\delta]$ and $a,b,c:{\CI}\times \T^2\to\R$ with $a(E,\cdot), \, b(E,\cdot), \, c(E,\cdot)\in C^{\omega}(\T^2,\R)$ for every $E\in{\CI}$,
 such that the conclusions of Theorem \ref{thm_Schro} and \ref{thm_Schro_sobolev} hold for Eq. (\ref{eq_AMO}). Moreover,  ${\rm Leb}(\Lambda_j)> 0$ for every $j\in\N$.
 \end{Theorem}

\begin{Remark} One can even further get precise size of ${\rm Leb}(\Lambda_j)$ according to \cite{LYZZ}.
\end{Remark}

\

The rest of paper will be organised as follows. In Section \ref{sec_weyl}, which serves as a preliminary section, we recall the definition of Weyl quantization and some known results on the relation between classical Hamiltonian to quantum Hamiltonian.
We give an abstract theorem in Section \ref{sec_abstract} on the reducibility for quantum Hamiltonian, provided that the reducibility for the corresponding classical Hamiltonian is known.
By applying this abstract theorem, we exploit the connection between reducibility and property of Sobolev norm.
The abstract theorem is proved in Section \ref{sec_reduc}.
In Section \ref{sec_proof}, we prove the main result just by verifying the hypothesis of abstract theorem.
In Section \ref{sec_pr_examples}, the proofs of Theorem \ref{thm_example_1} -- \ref{thm_AMO} are given.

\section{Classical Hamiltonian and quantum Hamiltonian}\label{sec_weyl}

To give some preliminary knowledge,
let us recall the definition of Weyl quantization, which relates the classical and quantum mechanics, and its properties. The conclusions listed in this section can also be found in \cite{BGMR2018}.

The Weyl quantization is the operator ${\rm Op}^W:f\mapsto f^W$ for any symbol $f=f(x,\xi)$, with $x,\xi\in\R^n$, where $f^{W}$ is the Weyl operator of $f$:
$$\left(f^{W} u\right)(x)=\frac{1}{(2\pi)^n}\int_{y, \, \xi\in\R^n} e^{{\rm i}(x-y)\xi} f\left(\frac{x+y}{2},\xi\right) u(y) \, dy  \, d\xi,\qquad \forall \ u\in L^2(\R^n).$$
In particular, if $f$ is a polynomial of degree at most $2$ in $(x,\xi)$, then $f^W$ is a polynomial of degree at most $2$ in $(x,-{\rm i}\partial_x)$ after the symmetrization.

For the $1-$parameter family of Hamiltonian $\chi(t, x, \xi )$, with $t$ an external parameter, let $\phi^\tau(t,x,\xi)$ be the time $\tau-$flow it generates, precisely the
solution of
$$\frac{dx}{d\tau}=\frac{\partial\chi}{\partial\xi}(t, x, \xi ),\qquad \frac{d\xi}{d\tau}=-\frac{\partial\chi}{\partial x}(t, x, \xi).$$
The time-dependent coordinate transformation
\begin{equation}\label{time1}
(x,\xi)=\phi^1\left(t,\tilde x,{\tilde\xi}\right)=\left.\phi^{\tau}\left(t,\tilde x,{\tilde\xi}\right)\right|_{\tau=1}
\end{equation}
transforms a Hamiltonian system with
Hamiltonian $h$ into a system with Hamiltonian $g$ given by
$$g(t,\tilde x,\tilde\xi)=h(\phi^1(t,\tilde x,\tilde\xi))-\int_0^1 \frac{\partial\chi}{\partial t}(t,\phi^{\tau}(t,\tilde x,\tilde\xi) ) d\tau.$$

\begin{Lemma} [Remark 2.6 of \cite{BGMR2018}] If the Weyl operator $\chi^W(t, x, -{\rm i}\partial_x)$ is self-adjoint for any fixed $t$, then the transformation
\begin{equation}\label{tran}
\psi=e^{{\rm i}\chi^W(t, x, -{\rm i}\partial_x)}\tilde\psi
\end{equation}
transforms the equation ${\rm i}\partial_t\psi=H\psi$ into ${\rm i}\partial_t\tilde\psi=G\tilde\psi$ with
\begin{eqnarray*}
G&:=&e^{{\rm i}\chi^W(t, x, -{\rm i}\partial_x)}He^{-{\rm i}\chi^W(t, x, -{\rm i}\partial_x)}\\
& & - \, \int_0^1 e^{{\rm i}\tau\chi^W(t, x, -{\rm i}\partial_x)}\left(\partial_t \chi^W(t, x, -{\rm i}\partial_x)\right)e^{-{\rm i}\tau\chi^W(t, x, -{\rm i}\partial_x)}d\tau.
\end{eqnarray*}

\end{Lemma}

\begin{Proposition} [Proposition 2.9 of \cite{BGMR2018}]\label{Prop_hami} Let $\chi(t, x, \xi )$ be a polynomial of degree at most $2$ in $(x,\xi)$ with smooth time-dependent coefficients.
If the transformation (\ref{time1}) transforms a classical system with Hamiltonian $h$ into
a system with Hamiltonian $g$, then the transformation (\ref{tran}) transforms the quantum Hamiltonian system
$h^W$ into $g^W$.
\end{Proposition}

Now, let us focus on the case $n=1$.

\begin{Lemma} [Lemma 2.8 of \cite{BGMR2018}]\label{lem_Sobolev}
Let $\chi(\theta,x,\xi)$ be a polynomial of degree at most $2$ in $(x,\xi)$ with real coefficients depending in a $C^\infty-$way on $\theta\in \T^d$.
For every $\theta\in \T^d$, the Weyl operator $\chi^W(\theta,x, -{\rm i}\partial_x)$ is self-adjoint in $L^2(\R)$ and $e^{-{\rm i}\tau\chi^W(\theta,x, -{\rm i}\partial_x)}$ is unitary in $L^2(\R^n)$ for every $\tau\in\R$.
Furthermore, if the coefficients of $\chi(\theta,x,\xi)$ are uniformly bounded w.r.t. $\theta\in \T^d$, then for any $s\geq 0$, there exist $c'$, $C' > 0$ depending on $\|[H_0^s,\chi^W(\theta,x, -{\rm i}\partial_x)]H_0^{-s}\|_{L^2\mapsto L^2}$ and $s$, such that
   \begin{equation}\label{change_Sobolevnorm}
      c'\|\psi\|_{s}\leq \|e^{-{\rm i}\tau\chi^W(\theta,x,-{\rm i}\partial_x)}\psi\|_{s}\leq C'\|\psi\|_{s},\qquad \tau\in [0,1], \quad \theta\in\T^d.
    \end{equation}
\end{Lemma}
%
%
%


\section{Reducibility and growth of Sobolev norm}\label{sec_abstract}

\subsection{An abstract theorem on reducibility}

Consider the 1-d time-dependent equation
\begin{equation}\label{eq_abs}
 {\rm i}\partial_t u=L^{W}(\omega t, x, -{\rm i} \partial_x)u,\qquad x\in\R ,
\end{equation}
where $L^{W}(\omega t, x, -{\rm i} \partial_x)$ is a linear differential operator, $\omega\in\T^d$, $d\geq 1$, and the symbol $L(\theta, x,\xi)$ is a quadratic form of $(x,\xi)$ with coefficients
analytically depending on $\theta\in\T^d$. More precisely, we assume that
\begin{equation}\label{op_L}
L(\theta, x,\xi)=\frac12 \big(a(\theta) x^2+ b(\theta) x\cdot \xi + b(\theta) \xi\cdot  x + c(\theta)  \xi^2\big),
\end{equation}
with coefficients $a,b,c\in C^\omega(\T^d,\R)$.

Through Weyl quantization, the reducibility for the time-dependent PDE can be related to the reducibility for the ${\rm sl}(2,\R)-$linear system $(\omega, \, A(\cdot))$:
$$X'=A(\omega t)X,\qquad A\in C^{\omega}(\T^d,{\rm sl}(2,\R)).$$
Given $A_1, A_2  \in C^{\omega}(\T^d,{\rm sl}(2,\R))$, if there exists $Y\in C^{\omega}(2\T^d,{\rm SL}(2,\R))$ such that
$$\frac{d}{dt}Y(\omega t)=A_1(\omega t)Y(\omega t)-Y(\omega t)A_2(\omega t),$$
 we say that $(\omega, \, A_1(\cdot))$ is conjugated to $(\omega, \, A_2(\cdot))$ by $Y$.
If $(\omega, \, A(\cdot))$ can be conjugated to $(\omega, \, B)$ with $B\in{\rm sl}(2,\R)$, we say that $(\omega, \, A(\cdot))$ is {\it reducible}.

\smallskip

Now let $A(\cdot):=\left(\begin{array}{cc}
            b(\cdot) & c(\cdot) \\[1mm]
           -a(\cdot) & -b(\cdot)
          \end{array}\right) \in C^{\omega}(\T^d,{\rm sl}(2,\R)) $ with $a,b,c$ coefficients given in (\ref{op_L}).
\begin{Theorem}\label{thm_redu}
Assume that there exist $B\in{\rm sl}(2,\R)$ and $Z_j\in C^\omega(2\T^d, {\rm sl}(2,\R))$, $j=0, \cdots,K$,
such that $(\omega, \, A(\cdot))$ is conjugated to $(\omega, \, B)$ by $\prod_{j=0}^K e^{Z_j}$. Then Eq. (\ref{eq_abs}) is reducible, i.e., there exists a time quasi-periodic map $U(\omega t)$, unitary in $L^2$ and analytic on $t$, satisfying
\begin{equation}\label{norm_U}
c'\|\psi\|_{s}\leq \|U(\omega t)\psi\|_{s}\leq C'\|\psi\|_{s},  \quad \forall \ \psi\in{\CH}^s,
\end{equation}
for constants $c', \, C'>0$ depending on $s$ and $\psi$, such that
Eq. (\ref{eq_abs}) is conjugated to
\begin{equation}\label{Ham_G}
{\rm i}\partial_t v= G v
\end{equation}
by the transformation $u=U(\omega t) v$, with $G$ an operator independent of time.

More precisely,
\begin{itemize}
\item [(\uppercase\expandafter{\romannumeral1})] $G$ is unitary equivalent to $\frac{\sqrt{{\rm det} B}}{2} H_0$ if
\begin{equation}\label{type1}
{\rm det}B>0 \  or  \ B=\left(\begin{array}{cc}
           0 & 0 \\[1mm]
           0 & 0
          \end{array}
\right).
\end{equation}

\item [(\uppercase\expandafter{\romannumeral2})]
$G$ is unitary equivalent to $-\frac{{\rm i}\sqrt{-{\rm det}B}}{2}(x\cdot \partial_x+ \partial_x \cdot x)$ if
\begin{equation}\label{type2}
{\rm det}B<0.\end{equation}

\item [(\uppercase\expandafter{\romannumeral3})]
$G$ is unitary equivalent to $-\frac{\kappa}{2} x^2$ if
\begin{equation}\label{type3}
B \ is \  similar  \  to \  \left(\begin{array}{cc}
            0 & 0\\[1mm]
            \kappa & 0
          \end{array}
\right) \  with \ \kappa\neq 0.
\end{equation}
\end{itemize}
\end{Theorem}

\subsection{Growth of Sobolev norm via reducibility}

As an corollary of Theorem \ref{thm_redu}, we have:

\begin{Theorem}\label{thm_sobo} Under the assumption of Theorem \ref{thm_redu}, we consider the solution $u(t)=u(t,\cdot)$ to Eq. (\ref{eq_abs}) with the non-vanishing initial condition $u(0)\in {\CH}^s$, $s\geq 0$. There exists $c, C>0$, depending on $s$ and $u(0)$, such that, for any $t\geq 0$,
\begin{itemize}
\item If (\ref{type1}) holds, then
$c \leq \|u(t)\|_{s}\leq C$.
\item If (\ref{type2}) holds, then
$c  e^{\sqrt{-{\rm det}B}st}  \leq \|u(t)\|_{s}\leq C  e^{\sqrt{-{\rm det}B}st}$.
\item If (\ref{type3}) holds, then
$
c |\kappa|^s t^s \leq \|u(t)\|_{s}\leq C   |\kappa|^s \sqrt{1+t^2}^{\frac{s}2}
$.
\end{itemize}
\end{Theorem}

According to (\ref{norm_U}), to precise the growth of Sobolev norms for the solution to Eq. (\ref{eq_abs}), it is sufficient to study the reduced quantum Hamiltonian $G(x, -{\rm i}\partial_x)$ obtained in (\ref{Ham_G}), or more simply, the unitary equivalent forms of types (\uppercase\expandafter{\romannumeral1})$-$(\uppercase\expandafter{\romannumeral3}) listed in Theorem \ref{thm_redu}.

\smallskip

If (\ref{type1}) holds,  then $G$ is unitary equivalent to $\frac{\sqrt{{\rm det}B}}{2} H_0$. Since the ${\CH}^s-$norm of $e^{-{\rm i}t\frac{\sqrt{{\rm det}B}}{2} H_0}\psi_0$ is conserved for any $\psi_0\in{\CH}^s$, the boundedness of Sobolev norm is shown.
We focus on the cases where (\ref{type2}) and (\ref{type3}) hold, in which the growth of Sobolev norm occurs.

\begin{Proposition} \label{prop6}
For the equation
\begin{equation}\label{eq_hyper}
\partial_t v(t,x)=-\frac\lambda2 x\cdot\partial_x v (t,x)-\frac\lambda2 \partial_x(x\cdot v (t,x)), \qquad \lambda>0,
\end{equation}
 with non-vanishing initial condition $ v (0, \cdot)= v_0\in {\CH}^s$, $s\geq 0$, there exist two constants $\tilde c, \, \tilde C>0$, depending on $s$, $\lambda$ and $ v _0$, such that the solution satisfies
\begin{equation}\label{bounds_hyper}
\tilde c e^{\lambda st} \leq \|\psi(t,\cdot )\|_{s}\leq   \tilde C e^{\lambda st}, \qquad \forall \ t\geq0.
\end{equation}
\end{Proposition}

\begin{Remark}
This conclusion is also given in Remark 1.4 of \cite{MR2017}.
\end{Remark}

\proof Through a straightforward computation, we can verify that, for the initial condition $ v (0,\cdot)= v _0(\cdot)\in {\CH}^s$, the solution to Eq. (\ref{eq_hyper}) satisfies
$$ v (t,x)=e^{-\frac\lambda2 t}  v _0(e^{-\lambda t} x).$$
For any $s\geq 0$,
\begin{eqnarray}
\int_\R x^{2s } |  v (t,x)|^2 \, dx &=& \int_\R  x^{2s }   | v _0(e^{-\lambda t} x)|^2 \, d (e^{-\lambda t}x) \nonumber\\
&=& e^{2\lambda s t}\int_\R (e^{-\lambda t} x)^{2s}| v _0(e^{-\lambda t} x)|^2 \, d (e^{-\lambda t}x)\nonumber\\
&=& e^{2\lambda s t}\int_\R x^{2s} | v _0(x)|^2 \, dx.\label{sobolev_hyper}
\end{eqnarray}
and for $s\in\N$,
\begin{equation}\label{sobolev_hyper-ds}
\int_\R |\partial_x^s  v (t,x)|^2 \, dx = e^{-2\lambda s t} \int_\R  | v _0^{(s)}(e^{-\lambda t} x)|^2 \, d (e^{-\lambda t}x)
= e^{-2\lambda s t}\int_\R | v _0^{(s)}(x)|^2 \, dx.
\end{equation}
In view of the equivalent definition (\ref{norm_equiv}) of the ${\CH}^s-$norm given in Remark \ref{remark_norm_equiv}, we get (\ref{bounds_hyper}) by combining (\ref{sobolev_hyper}) and (\ref{sobolev_hyper-ds}).\qed

\begin{Proposition}\label{prop_para}
For the equation
\begin{equation}\label{eq_para}
{\rm i}\partial_t v (t,x)=-\frac{\kappa}{2} x^2\cdot v (t,x), \qquad \kappa\in\R,
\end{equation}
with non-vanishing initial condition $ v _0\in {\CH}^s$, $s\geq 0$, there exists constants $\tilde c, \tilde C>0$, depending on $s$, $\kappa$ and $ v _0$, such that the solution satisfies
\begin{equation}\label{sobo_para}
\tilde c |\kappa|^s |t|^s \leq \| v (t,\cdot)\|_{s}\leq \tilde C  |\kappa|^s (1+ t^2)^\frac{s}2,\qquad \forall \  t\in\R.
\end{equation}
\end{Proposition}
\proof
With the initial condition $ v (0,\cdot)= v _0(\cdot)\in {\CH}^s$, the solution to Eq. (\ref{eq_para}) is
$$ v (t,x)=e^{{\rm i}\frac{\kappa}{2} x^2 t} v _0(x).$$
For any $s\geq 0$,
$$\|x^s v (t,x)\|_{L^2}=\|x^s e^{{\rm i}\frac{\kappa}{2} x^2 t} v _0(x)\|_{L^2}=\|x^s  v _0(x)\|_{L^2},$$
and for $s\in\N$,
\begin{eqnarray*}
\partial_{x}^{s}( v (t,x))
&=&\partial_{x}^{s}(e^{{\rm i}\frac{\kappa}{2} x^2 t} v _0(x))\\
&=& \sum_{\alpha=0}^{s} C_s^\alpha (e^{{\rm i}\frac{\kappa}{2} x^2 t})^{(\alpha)} v _0^{(s-\alpha)}(x)\\
&=& e^{{\rm i}\frac{\kappa}{2} x^2 t}  \sum_{\alpha=0}^{s} C_s^\alpha \left(({\rm i}\kappa t)^{\alpha} x^{\alpha}+P_{\alpha}({\rm i}\kappa t,x)\right)  v _0^{(s-\alpha)}(x) \\
&=&({\rm i} \kappa t)^{s}  x^{s}  e^{{\rm i}\frac{\kappa}{2} x^2 t}\cdot  v _0(x) +P_{s}({\rm i}\kappa t,x)e^{{\rm i}\frac{\kappa}{2} x^2 t}\cdot  v _0(x)\\
& &+ \, x^{\alpha} e^{{\rm i}\frac{\kappa}{2} x^2 t}  \sum_{\alpha=0}^{s-1} C_s^\alpha \left(({\rm i}\kappa t)^{\alpha} x^{\alpha}+P_{\alpha}({\rm i}\kappa t,x)\right)  v _0^{(s-\alpha)}(x),
\end{eqnarray*}
where, for $\alpha\geq 2$, $P_{\alpha}({\rm i}\kappa t,x)$ is a polynomial of degree $\alpha-2$ of $x$, with the coefficients being monomials of ${\rm i}\kappa t$ of degree $\leq \alpha-1$ and $P_{1}=P_0=0$. Then, there exists a constant $D>0$ such that
$$\left|\|\partial_{x}^{s}( v (t,x))\|_{L^2}-|\kappa t |^{s}\|x^{s}   v _0(x)\|_{L^2}\right|\leq D |\kappa t |^{s-1} \| v _0(x)\|_s.$$
In view of the equivalent definition (\ref{norm_equiv}) of norm in Remark \ref{remark_norm_equiv},
we get (\ref{sobo_para}). \qed

\smallskip

\noindent{\bf Proof of Theorem \ref{thm_sobo}.}
From Theorem \ref{thm_redu}, we know that
Eq. (\ref{eq_abs}) is conjugated to
${\rm i}\partial_t v= G v$ by the transformation $u=U(\omega t) v$, with $G=G(x,-{\rm i}\partial_x)$ the operator independent of $t$ given in (\ref{Ham_G_pr}).

Recall Proposition \ref{prop6} and \ref{prop_para}.
Given $s\geq 0$, for any non-vanishing $v_0\in{\CH}^s$, for the three types of unitary equivalence of $G$, there are three different behaviours of the solution to the equation ${\rm i}\partial_t v= G v$ as $t\to  \infty$.
\begin{itemize}
\item If $G$ is unitary equivalent to $\frac{\sqrt{{\rm det} B}}{2} H_0$ (under (\ref{type1})), then
$\|e^{-{\rm i}Gt}v_0\|_s=O(1)$.
\item If $G$ is unitary equivalent to $-\frac{{\rm i}\sqrt{-{\rm det} B}}{2} (x\cdot \partial_x +\partial_x\cdot x)$ (under (\ref{type2})), then
$\|e^{-{\rm i}Gt}v_0\|_s=O(e^{\sqrt{-{\rm det} B} s t}).$
\item If $G$ is unitary equivalent to $-\frac{\kappa}{2} x^2$ (under (\ref{type3})), then $\|e^{-{\rm i}Gt}v_0\|_s=O(|\kappa|^st^s)$.
\end{itemize}
Moreover, according to (\ref{norm_U}), for $s\geq 0$, there exist constants $c', C'>0$ such that
$$  c'\|v\|_{s}\leq \|U(\omega t)v\|_{s}\leq C'\|v\|_{s},\qquad  \forall \  v\in {\CH}^s.$$
Hence Theorem \ref{thm_sobo} is shown.\qed

\section{Reducibility in classical Hamiltonian system and proof of Theorem \ref{thm_redu}}\label{sec_reduc}

\subsection{Conjugation between classical hamiltonians}

Given two quadratic classical Hamiltonians
$$h_j(\omega t, x, \xi)=\frac12 \big(a_j(\omega t)x^2+ 2b_j(\omega t) x\cdot \xi+ c_j(\omega t)  \xi^2 \big), \qquad j=1,2,$$
which can be presented as
$$h_j(\omega t, x,\xi)=\frac12\left(
                             \begin{array}{c}
                               x \\
                               \xi \\
                             \end{array}
                           \right)^{\top}J A_j(\omega t)\left(
                             \begin{array}{c}
                               x \\
                               \xi \\
                             \end{array}
                           \right), \qquad j=1,2$$
with $J:=\left(\begin{array}{cc}
0 & -1 \\1 & 0 \end{array}\right)$ and
$A_j(\cdot)=\left(\begin{array}{cc}
b_j(\cdot) & c_j(\cdot) \\ -a_j(\cdot) & -b_j(\cdot) \end{array}\right)\in C^{\omega}(\T^d, {\rm sl}(2,\R))$.
The corresponding equations of motion are  given by
$$x'=\frac{\partial h_j}{\partial\xi},\quad \xi'=-\frac{\partial h_j}{\partial x},\qquad j=1,2,$$
which are the linear systems $(\omega, \, A_j)$:
$$\left(\begin{array}{c}
          x(t) \\
          \xi(t)
        \end{array}
\right)'=A_j(\omega t)\left(\begin{array}{c}
          x(t) \\
          \xi(t)
        \end{array}
\right).$$


\begin{Proposition}\label{prop_ham_cl}
If the linear system $(\omega, \, A_1(\cdot))$ is conjugated to $(\omega, \, A_2(\cdot))$ by a time quasi-periodic ${\rm SL}(2,\R)-$transformation, i.e.,
\begin{equation}\label{conj_ode}
\frac{d}{dt} e^{Z(\omega t)}=A_1(\omega t)e^{Z(\omega t)}-e^{Z(\omega t)} A_2(\omega t),\qquad Z \in C^\omega(2\T^d, {\rm sl}(2,\R)),
\end{equation}
then the classical Hamiltonian $h_1(\omega t,x,\xi)$ is conjugated to $h_2(\omega t,x,\xi)$ via the time$-1$ flow $\phi_{\chi}^1(t,x,\xi)$
 generated by the Hamiltonian
 \begin{equation}\label{chi_eZ}
 \chi(\omega t,x,\xi)=\frac12\left(
                             \begin{array}{c}
                               x \\
                               \xi \\
                             \end{array}
                           \right)^{\top}J Z(\omega t)\left(
                             \begin{array}{c}
                               x \\
                               \xi \\
                             \end{array}
                           \right).
 \end{equation}
\end{Proposition}

\begin{Remark} This is actually a consequence of the fact that in dimension $2$(one degree
of freedom), canonical transformations coincide with the transformations which preserve
the volume.
\end{Remark}

\proof Note that the equation of motion of the classical Hamiltonian $h_1$ is the linear system $(\omega, \, A_1(\cdot))$:
$$ \left(\begin{array}{c}
   x \\
    \xi
  \end{array}\right)'=A_1(\omega t)\left(\begin{array}{c}
   x \\
    \xi
  \end{array}\right).$$
In view of (\ref{conj_ode}), the transformation
\begin{equation}\label{tramSL}
  \left(\begin{array}{c}
   x \\
    \xi
  \end{array}\right)= e^{Z(\omega t)}\left(\begin{array}{c}
    \tilde x \\
   \tilde \xi
  \end{array}\right),\qquad Z\in C^{\omega}(2\T^d, {\rm sl}(2,\R)),
\end{equation}
conjugates $(\omega, \, A_1(\cdot))$ to $(\omega, \, A_2(\cdot))$. More precisely,
\begin{eqnarray*}
\left(\begin{array}{c}
\tilde x \\
\tilde \xi
\end{array}\right)'&=&e^{-Z(\omega t)}A_1(\omega t)\left(\begin{array}{c}
x \\
\xi
\end{array}\right)-e^{-Z(\omega t)}\frac{d}{dt}e^{Z(\omega t)}\left(\begin{array}{c}
\tilde x\\
\tilde \xi
\end{array}\right)  \\
&=& e^{-Z(\omega t)}A_1(\omega t)e^{Z(\omega t)}\left(\begin{array}{c}
{\tilde x} \\
{\tilde \xi}
\end{array}\right)-e^{-Z(\omega t)}\frac{d}{dt}e^{Z(\omega t)}\left(\begin{array}{c}
\tilde x\\
\tilde \xi
\end{array}\right)\\
&=&A_2(\omega t)\left(\begin{array}{c}
\tilde x\\
\tilde \xi
\end{array}\right),
\end{eqnarray*}
for which the corresponding Hamiltonian is $h_2(\omega t,\tilde x,\tilde\xi)$.
As in (3-35) of \cite{BGMR2018}, the time$-1$ map between the two Hamiltonians is generated by (\ref{chi_eZ}) since there is only quadratic terms in the Hamiltonian in our case.\qed

\subsection{Proof of Theorem \ref{thm_redu}}

We consider the classical Hamiltonian
\begin{eqnarray*}
L(\omega t,x,\xi)&=&\frac{a(\omega t)}{2}x^2+\frac{b(\omega t)}{2}(x\cdot\xi+\xi \cdot x)+\frac{c(\omega t)}{2}\xi^2 \nonumber \\
&=&\frac12X^{\top}J A(\omega t)X,\qquad X:= \left(
                             \begin{array}{c}
                               x \\
                               \xi \\
                             \end{array}
                           \right).
\end{eqnarray*}
with $a,b,c\in C^{\omega}(\T^d)$ given in Eq. (\ref{eq_abs}), and $A:=\left(\begin{array}{cc}
            b & c \\
           -a & -b
          \end{array}
\right)\in C^{\omega}(\T^d, {\rm sl}(2,\R))$.

By the hypothesis of Theorem \ref{thm_redu}, the linear system $(\omega, \, A(\cdot))$ can be reduced to the constant system $(\omega, \, B)$, with $B=\left(\begin{array}{cc}
            B_{11} & B_{12} \\
            -B_{21} & -B_{11}
          \end{array}
\right)\in{\rm sl}(2,\R)$, via finitely many transformations $(e^{Z_j})_{j=0}^K$ with $Z_j\in C^{\omega}(2\T^d, {\rm sl}(2,\R))$. Hence the reduced classical Hamiltonian is
$$g(x,\xi)=\frac12X^{\top}J B X= \frac{B_{21}}2 x^2+\frac{B_{11}}{2}(x\cdot \xi+ \xi\cdot x)+\frac{B_{12}}2 \xi^2.$$
By Proposition \ref{Prop_hami}, we see that $L^W(\omega t, x, -{\rm i}\partial_x)$ is conjugated to
\begin{equation}\label{Ham_G_pr}
G(x, -{\rm i}\partial_x):=g^W(x, -{\rm i}\partial_x)=\frac{B_{21}}2 x^2-\frac{B_{11}}{2}(x\cdot{\rm i}\partial_x+{\rm i}\partial_x\cdot x)-\frac{B_{12}}2 \partial_x^2
\end{equation}
via the product of unitary (in $L^2(\R)$) transformations
$$
U(\omega t):= \prod_{j=0}^K e^{-{\rm i}\chi^W_j(\omega t,x,-{\rm i}\partial_x)}
$$
where $\chi^W_j$ is the Weyl quantization of
$$\chi_j(\omega t,x,\xi)=\frac12X^{\top} J Z_j(\omega t)X.$$
Then (\ref{norm_U}) is deduced from (\ref{change_Sobolevnorm}) in Lemma \ref{lem_Sobolev}.
The following diagram gives a straightforward explanation for the above proof.
$$
\begin{array}{rcccl}
& X'=A(\omega t)X &\stackrel{\prod_{j=0}^K e^{Z_j(\omega t)}}{\longrightarrow} & X'=BX &  \  \  Z_j\in C^\omega(2\T^d, {\rm sl}(2,\R))  \\
  &   &  & &  \\
  &  \big\updownarrow &  &  \big\updownarrow & \\
  &  &  &  & \\
 & L(\omega t)=\frac12X^{\top} J A(\omega t)X & \stackrel{ \Phi^1_{\chi_0(\omega t)}\circ \cdots \circ  \Phi^1_{\chi_K(\omega t)}}{\longrightarrow} & g=\frac12X^{\top} J B X & \ \ \chi_j =\frac12X^{\top}J Z_j X\\
  &  &  & &  \\
   &  \big\updownarrow &  &  \big\updownarrow & \\
  &  &  & &  \\
 & {\rm i}\partial_t u=L^W(\omega t)u &  \stackrel{\prod_{j=0}^{K} e^{-{\rm i} \chi_j^W(\omega t)}}{\longrightarrow} &  {\rm i}\partial_t u = g^W u &
  \end{array}
$$

If (\ref{type1}) holds, i.e., ${\rm det}B>0$ or $B=\left(\begin{array}{cc}
            0 & 0 \\
            0 & 0
          \end{array}\right)$,
then there exists $C_B\in {\rm sl}(2,\R)$ such that
\begin{equation}\label{elliptic}
B=e^{C_B}\left(\begin{array}{cc}
            0 & \sqrt{{\rm det}B} \\
            -\sqrt{{\rm det}B} & 0
          \end{array}\right)e^{-C_B}.
\end{equation}
If (\ref{type2}) holds,  i.e., ${\rm det}B<0$, then there exists $C_B\in {\rm sl}(2,\R)$ such that
\begin{equation}\label{hyerbolic}
B=e^{C_B}\left(\begin{array}{cc}
            \sqrt{-{\rm det}B} & 0 \\
            0 & -\sqrt{-{\rm det}B}
          \end{array}\right)e^{-C_B}.
\end{equation}
If (\ref{type3}) holds,
then there exists $C_B\in {\rm sl}(2,\R)$ such that
\begin{equation}\label{parapolic}
B=e^{C_B}\left(\begin{array}{cc}
            0 & 0 \\
            \kappa & 0
          \end{array}\right)e^{-C_B}.
\end{equation}
Therefore, for Eq. (\ref{eq_abs}), the three types of unitary equivalence of $G=G(x,-{\rm i}\partial_x)$ are shown by (\ref{elliptic})$-$(\ref{parapolic}) respectively. \qed

\section{Proof of Theorem \ref{thm_Schro} and \ref{thm_Schro_sobolev}}\label{sec_proof}

In view of Theorem \ref{thm_redu}, to show the reducibility of Eq. (\ref{eq_Schrodinger}), it is sufficient to show the reducibility of the corresponding ${\rm sl}(2,\R)-$linear system.

For $E\in{\CI}$, the symbol of the quantum Hamiltonian (\ref{eq_Schrodinger}) is
$$h_E(\omega t, x,\xi)=\frac{\nu(E)}{2}(\xi^2+x^2)+W(E,\omega t,x,\xi)$$
which corresponds the quasi-periodic linear system $(\omega, \,  A_0+F_0)$
\begin{equation}\label{linear_system_pr}
\left(\begin{array}{c}
          x \\
          \xi
        \end{array}
\right)'=\left[\left(\begin{array}{cc}
             0 & \nu(E) \\
             -\nu(E) & 0
           \end{array}\right)
           +\left(\begin{array}{cc}
             b(E,\omega t) & c(E,\omega t) \\
             -a(E,\omega t) & -b(E,\omega t)
           \end{array}\right)\right] \left(\begin{array}{c}
          x \\
          \xi
        \end{array}
\right),
\end{equation}
where, for every $E\in{\CI}$,
\begin{eqnarray*}
A_0(E)&:= & \left(\begin{array}{cc}
             0 & \nu(E) \\
             -\nu(E) & 0
           \end{array}\right)\in {\rm sl}(2,\R), \\
 F_0(E,\cdot)&:=& \left(\begin{array}{cc}
             b(E,\cdot) & c(E,\cdot) \\
             -a(E,\cdot) & -b(E,\cdot)
           \end{array}\right)\in C^{\omega}_r(\T^d,{\rm sl}(2,\R))
\end{eqnarray*}
with $|\partial_E^m F_0|_r<\varepsilon_0$, $m=0,1,2$, sufficiently small.

The reducibility of linear system (\ref{linear_system_pr}) was proved by Eliasson \cite{Eli1992} (see also \cite{HA} for results about ${\rm SL}(2,\R)$-cocycles). We summarise the needed results in the following proposition. To make the paper as self-contained as possible, we give a short proof without adding too many details on known facts.
Since every quantity depends on $E$, we do not always write this dependence explicitly in the statement of proposition.

Before stating the precise result, we introduce the concept of rotation number. The {\it rotation number}
of quasi-periodic ${\rm sl}(2,\R)-$linear system (\ref{linear_system_pr})
 is defined as
$$\rho(E)=\rho(\omega, \, A_0(E)+F(E,\omega t))=\lim_{t\to\infty}\frac{\arg(\Phi_E^t X)}{t},\quad \forall \ X\in \mathbb{R}^2\setminus\{0\},$$
where $\Phi_E^t$ is the basic
matrix solution and $\arg$ denotes the angle. The rotation number
$\rho$ is well-defined and it does not depend on $X$
\cite{JM82}.

\begin{Proposition}\label{prop_eliasson} There exists $\varepsilon_*=\varepsilon_*(r,\gamma,\tau,d,l_1,l_2)>0$ such that if
\begin{equation}\label{small_F_0}
\max_{m=0,1,2}|\partial_E^m F_0|_r=:\varepsilon_0<\varepsilon_*,
\end{equation}
then the following holds for the quasi-periodic linear system $(\omega, \, A_0+F_0)$.
\begin{enumerate}
  \item [(1)] For a.e. $E\in{\CI}$, $(\omega, \, A_0+F_0(\cdot))$ is reducible. More precisely,
 there exist $B\in{\rm sl}(2,\R)$ and  $Z_j\in C^\omega(2\T^d, {\rm sl}(2,\R))$, $j=0,1,\cdots,K$, such that
 \begin{equation}\label{reducibility_sl2R}
   \frac{d}{dt}\left(\prod_{j=0}^K e^{Z_j(\omega t)}\right)=\left(A_0+F_0(\omega t)\right)\left(\prod_{j=0}^K e^{Z_j(\omega t)}\right)-\left(\prod_{j=0}^K e^{Z_j(\omega t)}\right)B.
 \end{equation}
  \item [(2)] The rotation number $\rho=\rho(E)$ is monotonic on $\CI$. For any $k\in\Z^d$,
  $$\tilde\Lambda_k:=\left\{E\in\overline{\CI}:\rho(E)=\frac{\la k,\omega\ra}{2}\right\}  \   \footnote{$\tilde\Lambda_k$ can be empty for some $k\in\Z^d$ if the closed interval $\rho^{-1}\left(\frac{\la k,\omega\ra}{2}\right)$ does not intersect ${\CI}$.} $$ is a closed interval, and we have
\begin{equation}\label{measure_esti}
\sum_{k\in\Z^d}{\rm Leb}(\tilde\Lambda_k)<\varepsilon_0^{\frac{1}{40}}.
\end{equation}
  \item [(3)] For every $E\in \tilde\Lambda_k=:[a_k,b_k]$, $(\omega, \, A_0+F_0(\cdot))$ is reducible and the matrix $B\in {\rm sl}(2,\R)$ in (\ref{reducibility_sl2R}) satisfies
 \begin{itemize}
   \item if $a_k=b_k$, then $B=\left(\begin{array}{cc}
                                       0 & 0 \\
                                       0 & 0
                                     \end{array}\right)$;
   \item if $a_k<b_k$, then
\begin{itemize}
  \item ${\rm det} B<0$ for $E\in(a_k,b_k)$,
  \item ${\rm det} B=0$ for $E=a_k,b_k$ and $E\not\in\partial {\CI}$.
\end{itemize}
 \end{itemize}
   \item [(4)] For a.e. $E\in{\CI}\setminus\bigcup_k \tilde\Lambda_k$, $(\omega, \, A_0+F_0(\cdot))$ is reducible and the matrix $B\in {\rm sl}(2,\R)$ in (\ref{reducibility_sl2R}) satisfies ${\rm det} B>0$.
\end{enumerate}
\end{Proposition}

\proof Since $\nu$ is a strictly monotonic real-valued function of $E\in{\CI}$ and $|\nu'|\geq l_1$, $|\nu''|\leq l_2$, (\ref{small_F_0}) implies that
$|\partial^m_E F_0(\nu^{-1}(E),\cdot)|_r$, $m=0,1,2$, is also small enough.
 Hence, to prove the above arguments, we can simply consider the case where $\nu(E)=E\in {\CI}= \R$ and then obtain Proposition \ref{prop_eliasson} by replacing $E$ by $\nu(E)$.

\smallskip

\noindent {\it Proof of (1).} The reducibility has already been shown by Eliasson \cite{Eli1992} for a.e. $E\in\R$.
Indeed, if $\max_{m=0,1,2}|\partial_E^m F_0|_r$ is small enough (depending on $r,\gamma,\tau,d$), then there exists sequences $(Y_j)_{j\in\N}\subset C^\omega(2\T^d, {\rm SL}(2,\R))$, $(A_j)_{j\in\N}\subset {\rm sl}(2,\R)$, and $(F_j)_{j\in\N}\subset C^\omega(2\T^d, {\rm sl}(2,\R))$, all of which are piecewise $C^2$ w.r.t. $E$,
 with $\max_{m=0,1,2}|\partial^m_E F_j|_{\T^d}<\varepsilon_j:=\varepsilon_0^{(1+\sigma)^j}$ for $\sigma=\frac1{33}$, such that
$$\frac{d}{dt}Y_{j}(\omega t)=\left(A_j+F_j(\omega t)\right)Y_{j}(\omega t)-Y_{j}(\omega t) \left(A_{j+1}+F_{j+1}(\omega t)\right).$$
More precisely, at the $j-$th step, for $\pm{\rm i}\xi_j\in\R\cup{\rm i}\R$, the two eigenvalues of $A_j$, and
$$N_j:=\frac{2\sigma}{r_j-r_{j+1}}\ln\left(\frac{1}{\varepsilon_j}\right)$$
with $(r_j)_{j\in\N}$ a decreasing sequence of positive numbers such that $r_j-r_{j+1}\geq 2^{-(j+1)}r$ for each $j$,
\begin{itemize}
  \item (non-resonant case) if for every $n\in\Z^d$ with $0<|n|\leq N_j$, we have
  \begin{equation}\label{non_resonant}
  \left|2\xi_j-\la n,\omega\ra\right|\geq \varepsilon_j^{\sigma},
  \end{equation}
  then $Y_{j}=e^{\tilde Z_{j}}$ for some $\tilde Z_{j}\in C^{\omega}(2\T^d,{\rm sl}(2,\R))$ with $|\tilde Z_{j}|_{2\T^d}<\varepsilon_j^{\frac23}$, and $|A_{j+1}-A_j|<\varepsilon_j^{\frac23}$;
  \item (resonant) if for some $n_j\in\Z^d$ with $0<|n_j|\leq N_j$, we have
  \begin{equation}\label{resonant}
  \left|2\xi_j-\la n_j,\omega\ra\right|< \varepsilon_j^{\sigma},
  \end{equation}
  then $Y_{j+1}(\cdot)=e^{\frac{\la n_j ,\cdot\ra}{2\xi_j}A_j} e^{\tilde Z_{j+1}}$ for some $\tilde Z_{j}\in C^{\omega}(2\T^d,{\rm sl}(2,\R))$ with $|\tilde Z_{j}|_{2\T^d}<\varepsilon_j^{\frac23}$ and $|A_{j+1}|<\varepsilon_j^{\frac{\sigma}2}$.
\end{itemize}
As $j$ goes to $\infty$, the time-dependent part $F_{j}$ tends to vanish.
For the detailed proof, we can refer to Lemma 2 of \cite{Eli1992} and its proof.

In view of Lemma 3 b) of \cite{Eli1992}, if the rotation number $\rho(E)$ of $(\omega, \, A_0(E)+F_0)$ is Diophantine or rational w.r.t. $\omega$, which corresponds to a.e. $E\in\R$, then the resonant case occurs for only finitely many times.
Therefore, for a.e. $E\in\R$, there exists a large enough $J_*\in\N^*$, depending on $E$, such that
\begin{equation}\label{J_large}
Y_{j}=e^{\tilde Z_{j}} \  \ {\rm with} \   \  |\tilde Z_{j}|_{2\T^d}<\varepsilon^{\frac23}_{j},\qquad \forall \ j\geq J_*.
\end{equation}
This implies that $\prod_{j=0}^\infty|Y_j|_{2\T^d}$ is convergent.
As explained in the proof of Lemma 3.5 of \cite{BGMR2018}, (\ref{J_large}) also implies that there exists $S\in C^{\omega}(2\T^d,{\rm sl}(2,\R))$ such that
$\prod_{j=J_*}^\infty Y_{j}=e^S$,
since $\varepsilon_0$ is sufficiently small.
Hence (\ref{reducibility_sl2R}) is shown, i.e., the reducibility is realized via finitely many transformations of the form $e^{Z_j(\omega t)}$ with $Z_j\in C^{\omega}(2\T^d,{\rm sl}(2,\R))$.


\

\noindent {\it Proof of (2).}
For $k\in\Z^d$, $\tilde\Lambda_k$ is obtained after several resonant KAM-steps, saying $j_1$, $\cdots$, $j_L$, where $n_{j_i}\in\Z^d$ with $0<|n_{j_i}|\leq N_{j_i}$, $i=1,\cdots,L$, satisfies
$$ \left|2\xi_{j_i}-\la n_{j_i},\omega\ra\right|< \varepsilon_{j_i}^{\sigma},$$
and $k=n_{j_1}+\cdots + n_{j_L}$. We will show that
\begin{equation}\label{k_n_j_L}
\frac{10|n_{j_L}|}{11}\leq |k|\leq  \frac{12|n_{j_L}|}{11}.
\end{equation}
 Assume that $L\geq 2$ (otherwise we have already $k=n_{j_L}$). After the $(j_{i-1}+1)-$th step, $i=2,\cdots,L$, the eigenvalues $\pm{\rm i}\xi_{j_{i-1}+1}$ satisfies $|\xi_{j_{i-1}+1}|<2\varepsilon_{j_{i-1}}^{\frac\sigma2}$. On the other hand, before the $(j_{i}+1)-$th step, the resonant condition (\ref{resonant}) implies that the eigenvalues $\pm{\rm i}\xi_{j_{L}}$ satisfy that
$$|2\xi_{j_{i}}-\la n_{j_i},\omega\ra|\leq \varepsilon_{j_{i}}^\sigma.$$
Since the steps between these two successive resonant steps are all non-resonant, and $\omega\in {\rm DC}_{d}(\gamma,\tau)$,
we have that
$$\frac{\gamma}{|n_{j_i}|^\tau}\leq|\la n_{j_i},\omega\ra|\leq 2|\xi_{j_{i-1}+1}|+2\varepsilon^{\frac13}_{j_{i-1}+1}+\varepsilon_{j_{i}}^\sigma<3\varepsilon_{j_{i-1}}^{\frac\sigma2},$$
which implies that
$$|n_{j_i}|>\left(\frac\gamma3\right)^{\frac{1}{\tau}}\varepsilon^{-\frac{\sigma}{2\tau}}_{j_{i-1}}>12|N_{j_{i-1}}|\geq 12|n_{j_{i-1}}|.$$
Hence, we get (\ref{k_n_j_L}).

$\tilde\Lambda_k$ is firstly formed at the $j_L-$th step, with the initial measure smaller than $\varepsilon_{j_L}^{2\sigma}$.
Since all the succedent steps are non-resonant, the measure of $\tilde\Lambda_k$ varies up to $\varepsilon_{j_L}^{2\sigma}$. Then, for $\varsigma:=\frac{\ln(1+\sigma)}{\ln(8+8\sigma)}$, we have
$$
{\rm Leb}(\tilde\Lambda_k)< 2\varepsilon_{j_L}^{2\sigma}< 2\varepsilon_0^{\sigma} e^{-\left(\frac{12}{11}\right)^\varsigma N_{j_L}^\varsigma}\leq 2\varepsilon_0^{\sigma} e^{-\left(\frac{12}{11}\right)^\varsigma |n_{j_L}|^\varsigma}.
$$
Indeed, recalling that $r_j-r_{j+1}\geq 2^{-(j+1)}r$ for every $j$, we have
\begin{eqnarray*}
\varepsilon_{j_L} &=& \exp\{-|\ln\varepsilon_0|(1+\sigma)^{j_L}\} \\
 &=& \exp\left\{- \frac{|\ln\varepsilon_0|^{1-\varsigma}(1+\sigma)^{j_L(1- \varsigma)}(r_{j_L}-r_{j_L+1})^{\varsigma}}{(2\sigma)^{\varsigma}} N_{j_L}^\varsigma\right\}\\
 &\leq&\exp\left\{- \frac{|\ln\varepsilon_0|^{1-\varsigma} r^\varsigma}{(4\sigma)^{\varsigma}} \left(\frac{(1+\sigma)^{1- \varsigma}}{2^\varsigma}\right)^{j_L} N_{j_L}^\varsigma\right\}\\
 &<&\exp\left\{-\left(\frac{12}{11}\right)^\varsigma \frac{N_{j_L}^\varsigma}{\sigma}\right\},
\end{eqnarray*}
since $\varepsilon_0$ is small enough and
$$\frac{(1+\sigma)^{1- \varsigma}}{2^\varsigma}=\exp\left\{\frac{\ln(1+\sigma)}{\ln(8+8\sigma)}\left(\ln 8-\ln 2\right)\right\}>1.$$
Therefore, by (\ref{k_n_j_L}), we get
${\rm Leb}(\tilde\Lambda_k)<2\varepsilon_0^{\sigma} e^{- |k|^\varsigma}$,
which implies (\ref{measure_esti}). For detailed proof of the measure estimate of $\tilde\Lambda_k$, we can also refer to Corollary 1 of \cite{HA}.

\

\noindent {\it Proof of (3) and (4).} It can be deduced from Lemma 5 of \cite{Eli1992}. \qed

\

\noindent{\bf Proof of Theorem \ref{thm_Schro} and \ref{thm_Schro_sobolev}.} Theorem \ref{thm_Schro_sobolev} can be seen as a corollary of Theorem \ref{thm_sobo}.
According to Theorem \ref{thm_redu}, the reducibility of Eq. (\ref{eq_Schrodinger}) for a.e. $E\in{\CI}$ is deduced from Proposition \ref{prop_eliasson}-(1).
Let $\{\Lambda_j\}_{j\in\N}$ be the intervals $\tilde\Lambda_k$'s intersecting ${\CI}$ and let
$${\CO}_{\varepsilon_0}:=\bigcup_{j\in\N}\Lambda_j=\bigcup_{k\in\Z^d}\tilde\Lambda_k.$$
Proposition \ref{prop_eliasson}-(2) gives the measure estimate of ${\CO}_{\varepsilon_0}$.
The unitary equivalences of the reduced quantum Hamiltonian follow from Proposition \ref{prop_eliasson}-(3) and (4). Hence Theorem \ref{thm_Schro} is shown. \qed

\section{Proof of Theorem \ref{thm_example_1} -- \ref{thm_AMO}}\label{sec_pr_examples}

In this section, we show that the measure of the subset ${\CO}_{\varepsilon_0}$ is positive for the equations (\ref{example_1}) -- (\ref{eq_AMO}),
which implies the growths of Sobolev norm.

\subsection{Proof of Theorem \ref{thm_example_1}}

For Eq. (\ref{example_1}), $E\in\R$, the corresponding linear system is
$$
\left(\begin{array}{c}
          x \\
          \xi
        \end{array}
\right)'=\left[\left(\begin{array}{cc}
             0 & E \\
             -E & 0
           \end{array}\right)
           +\left(\begin{array}{cc}
             b(\omega t) & c(\omega t) \\
             -a(\omega t) & -b(\omega t)
           \end{array}\right)\right] \left(\begin{array}{c}
          x \\
          \xi
        \end{array}
\right).$$
In view of Lemma 5 of \cite{Eli1992}, for ``generic" $a,b,c\in C^\omega(\T^d,\R)$, there is at least one non-degenerate $\tilde\Lambda_k$, $k\in\Z^d$.
More precisely, at the resonant step of KAM scheme described in the proof of Proposition \ref{prop_eliasson}-(1),
the condition (\ref{resonant}) defines a resonant interval of $E$, on which the two eigenvalues $\pm{\rm i}\xi_j$ of $A_j$ are purely imaginary since $\xi_j$ is bounded frow below. After this resonant step, the two new eigenvalues $\pm{\rm i}\xi_{j+1}$ of $A_{j+1}$ can be real or still purely imaginary for $E$ in this resonant interval, since $|\xi_{j+1}|$ is close to zero.
We say that $a,b,c\in C^\omega(\T^d,\R)$ are {\it generic} if, for at least one resonant step in the KAM scheme, the two new eigenvalues $\pm{\rm i}\xi_{j+1}$ become real on a sub-interval of the resonant interval.

\subsection{Proof of Theorem \ref{thm_example_schro}}

For Eq. (\ref{eq_Schrodinger-example}) with $E\in{\CI}=[E_0,E_1]$ with $E_0>0$ large enough, and $E_1<\infty$, Theorem \ref{thm_Schro} and \ref{thm_Schro_sobolev} hold.
The corresponding linear system $(\omega, \, A_0+F_0)$ of Eq. (\ref{eq_Schrodinger-example}) is
$$
\left(\begin{array}{c}
          x \\
          \xi
        \end{array}
\right)'=\left[\left(\begin{array}{cc}
             0 & \sqrt{E} \\
             -\sqrt{E} & 0
           \end{array}\right)
           +\frac{q(\omega t)}{2\sqrt{E}}\left(\begin{array}{cc}
             -1 & -1 \\
             1 & 1
           \end{array}\right)\right] \left(\begin{array}{c}
          x \\
          \xi
        \end{array}
\right).$$
Then, through the change of variables
$$\left(\begin{array}{c}
          x \\
          \xi
        \end{array}\right)=\frac{1}{2\sqrt{E}}\left(\begin{array}{cc}
           \sqrt{E} & -1 \\
           \sqrt{E} & 1
           \end{array}\right)\left(\begin{array}{c}
          \tilde x \\
          \tilde\xi
        \end{array}\right),$$
$(\omega, \, A_0+F_0)$ is conjugated to
$$
\left(\begin{array}{c}
          \tilde x \\
          \tilde\xi
        \end{array}\right)'=C^E_q(\omega t)\left(\begin{array}{c}
          \tilde x \\
          \tilde\xi
        \end{array}\right):=\left(\begin{array}{cc}
             0 & 1 \\
             -E+q(\omega t) & 0
           \end{array}\right)\left(\begin{array}{c}
          \tilde x \\
          \tilde\xi
        \end{array}\right).
$$
The quasi-periodic linear system $(\omega, \, C^E_q(\cdot))$ corresponds exactly to the eigenvalue problem of the quasi-periodic continuous Schr\"odinger operator ${\CL}_{\omega, q}$:
 $$({\CL}_{\omega, q}y)(t)=-y''(t)+q(\omega t) y(t).$$
By Gap labeling Theorem \cite{JM82}, if $\tilde\Lambda_k$ is not empty for $k\in\Z^d$, then it is indeed a ``spectral gap" of ${\CL}_{\omega, q}$ intersecting $[E_0,E_1]$, i.e., a connected component of $[E_0,E_1]\setminus\Sigma_{\omega, q}$ with $\Sigma_{\omega, q}$ denoting the spectrum of ${\CL}_{\omega, q}$.
In view of Theorem C of \cite{Eli1992}, for a generic potential $q$ (in the $|q|_r$-topology), for $E_0>0$ large enough, $[E_0,\infty[ \, \cap \, \Sigma_{\omega, q}$ is a Cantor set.
Hence there are infinitely many $\tilde\Lambda_k$'s satisfying ${\rm Leb}(\tilde\Lambda_k)>0$.

\subsection{Proof of Theorem \ref{thm_AMO}}

For Eq. (\ref{eq_AMO}) with $\nu(E)=\cos^{-1}(-\frac{E}{2})$, $E\in[-2+\delta,2-\delta]$ with $\delta>0$ a sufficiently small numerical constant (e.g. $\delta:=10^{-6}$), we can apply Theorem \ref{thm_Schro} and \ref{thm_Schro_sobolev},
if $a,b,c:[-2+\delta,2-\delta]\times\T^2\to{\rm  sl}(2,\R)$ are small enough as assumed in Theorem \ref{thm_Schro}.

For the quasi-periodic Schr\"odinger cocycle $(\alpha, \, S_E^\lambda)$
$$
X_{n+1} =S_E^\lambda(\theta+n\alpha) X_n= \left[\left(\begin{array}{cc}
             -E & -1 \\
             1 & 0
           \end{array}\right)+\left(\begin{array}{cc}
             2\lambda \cos(\theta+n\alpha) & 0 \\
             0 & 0
           \end{array}\right)\right] X_n,
$$
with $\alpha\in{\rm DC}_1(\gamma,\tau)$, $|\lambda|$ small enough,
it can be written as
$$X_{n+1}=e^{B(E)}e^{G(E,\theta+n\alpha)} X_n,$$
for $e^{B(E)}:=\left(\begin{array}{cc}
             -E & -1 \\
             1 & 0
           \end{array}\right)$ and some $G(E,\cdot)\in{\rm sl}(2,\R)$.
This cocycle is related to the almost-Mathieu operator $H_{\lambda,\alpha,\theta}$ on $\ell^2(\Z)$:
$$(H_{\lambda,\alpha,\theta}\psi)_n=-(\psi_{n+1}+\psi_{n-1})+2\lambda \cos(\theta+n\alpha)\psi_n,\qquad n\in\Z.$$
It is known that its spectrum, denoted by $\Sigma_{\lambda,\alpha}$, is a Cantor set \cite{AvilaJito1},  which is well-known as Ten Martini Problem. In fact, Avila-Jitomirskaya
 \cite{AvilaJito} further show that  all spectral gaps are ``open" , which means that, for every $k\in\Z$,
$$\tilde\Lambda_k:=\left\{E\in\R: \tilde\rho_{(\alpha, \, S_E^\lambda)}=\frac{k\alpha}{2} \ {\rm mod} \ \Z \right\}$$
has positive measure. Indeed, the size of $\tilde\Lambda_k$ decays exponentially with respect to $|k|$, as was shown in \cite{LYZZ}.
Here, $\tilde\rho_{(\alpha, \, S_E^\lambda)}$ is the fibered rotation number of cocycle $(\alpha, \, S_E^\lambda)$.
Recall that for any  $A:\T^d \to {\rm SL}(2,\R)$ is continuous and homotopic to the identity,   \textit{fibered rotation number} of
$(\alpha,A)$ is defined as
$$
\tilde\rho(\alpha,A)=\int \psi \, d \tilde{\mu}  \ {\rm mod} \  \Z
$$
where $\psi:\T^{d+1} \to \R$ is  lift of $A$ such that
$$
A(x) \cdot \left (\bm \cos 2 \pi y \\ \sin 2 \pi y \em \right )=u(x,y)
\left (\bm \cos 2 \pi (y+\psi(x,y)) \\ \sin 2 \pi (y+\psi(x,y)) \em \right),
$$
and $\tilde{\mu}$ is
invariant probability  measure of $(x,y) \mapsto (x+\alpha,y+\psi(x,y))$  (according to \cite{Her}, it does not depend on the choices of $\psi, \tilde\mu$).

Note that $(\alpha, \, S_E^\lambda)$ is a discrete dynamical system, however, with the help of Local Embedding Theorem (Theorem \ref{localemb-sl}), we can embed the cocycle $(\alpha, \, S_E^\lambda)$ into a quasi-periodic linear system $(\omega, \, B(E)+F(E,\cdot))$.
For an individual cocycle, the Local Embedding Theorem was already shown in \cite{YZ2013}.
Nevertheless, the crucial point here is that we really need a parameterized version of Local Embedding Theorem, that means the embedded system   $(\omega, \, B(E)+F(E,\cdot))$ should have smooth dependence on $E$.

To show the parameterized version of Local Embedding Theorem, let us first introduce more notations.
 Given $f \in C^2(\CI)$, define
$$|f|_{*}= \sum_
{0\leq m\leq 2} \sup_{E\in\CI}|f^{(m)}|.$$
For any $ f(E,\theta)=\sum_{k\in \Z^d}\widehat f_k(E)e^{2\pi  {\rm i}
\langle k,\theta\rangle}$ which is $C^2$ w.r.t. $E\in{\CI}$, $C^\omega$ w.r.t. $\theta\in\T^d$, denote
$$\|f\|_h:=\sum_{k\in \Z^d}|\widehat f_k(E)|_{*}e^{2\pi |k|h}<\infty,$$
and we denote by  $C_h^\omega( \CI \times \T^d,\C)$ all these functions with $\|f\|_h<\infty$. Then our result is the following:

\begin{Theorem}\label{localemb-sl}[Local Embedding Theorem]
Given $d\geq 2$, $h>0$ and $G\in C^\omega_{h}({\CI} \times \T^{d-1}, {\rm sl}(2,\R))$, suppose that $\mu\in \T^{d-1}$ such that $(1,\mu)$ is rationally independent. Then, for any $\nu\in C^2(\CI)$ satisfying
 \begin{equation}\label{varition}
 \sup_{E\in\CI}|\nu'(E)|\cdot |\CI|< \frac{1}{6},
 \end{equation}
there exist $\epsilon=\epsilon(|\nu|_{*},h,|\mu|)>0,$
$c=c(|\nu|_*,h,|\mu|)>0,$ and $F\in
C^\omega_{\frac{h}{1+|\mu|}}(  \CI \times \T^d,{\rm sl}(2,\R))$ such that the cocycle $(\mu,e^{2\pi \nu J}
e^{G(\cdot)})$ is the Poincar\'e map of linear system
\begin{eqnarray}\label{al-ref1}
\left(\begin{array}{c}x\\ \xi
 \end{array}
 \right)'=\left(\nu J+F(\omega t)\right)\left(\begin{array}{c}x\\ \xi
 \end{array}
 \right) , \qquad \omega=(1,\mu)
\end{eqnarray}
provided that $\|G\|_{h}<\epsilon.$ Moreover, we have $\|F\|_{\frac{h}{1+|\mu|}}\leq 2c \|G\|_{h}$.
\end{Theorem}
We postpone the proof of Theorem \ref{localemb-sl} to Appendix \ref{app_proof}.

\smallskip

Now let us show how we can apply Theorem \ref{localemb-sl} to finish the proof of Theorem \ref{thm_AMO}. First note the constant matrix $e^{B}$ can be rewritten as

$$e^{B}:=\left(\begin{array}{cc}
             -E & -1 \\
             1 & 0
           \end{array}\right) = M\left(\begin{array}{cc}
            \cos(\nu) & -\sin(\nu) \\
             \sin(\nu) & \cos(\nu)
           \end{array}\right)M^{-1} , $$
where
$$M:=\frac{1}{\sqrt{\sin(\nu)}}\left(\begin{array}{cc}
             \cos(\nu) & -\sin(\nu) \\
             1 & 0
           \end{array}\right), $$
recalling that
$$\cos(\nu(E))=-\frac{E}{2},\quad \sin(\nu(E))=\frac{\sqrt{4-E^2}}{2},\qquad E\in[-2+\delta,\ 2-\delta].$$
Hence, by noting
$$\left(\begin{array}{cc}
            \cos(\nu) & -\sin(\nu) \\
             \sin(\nu) & \cos(\nu)
           \end{array}\right)=\exp\left\{\left(\begin{array}{cc}
            0 & -\nu \\
             \nu & 0
           \end{array}\right)\right\},$$
          we see that $B$ can be written as
$B=M  \cdot ( \nu J )  \cdot  M^{-1}.$

\smallskip

For $\nu(E)=\cos^{-1}(-\frac{E}{2})$, there exists ${\CI}\subset [-2+\delta, 2-\delta]$ such that (\ref{varition}) is satisfied.
For example, we can take ${\CI}= ]-\frac{2}{\sqrt{37}},\frac{2}{\sqrt{37}}[$.
Therefore, according to Theorem \ref{localemb-sl}, for $\omega\in (1,\alpha)$, we have a quasi-periodic linear system $(\omega, \, B(E)+F(E,\cdot))$ from  the quasi-periodic Schr\"odinger cocycle $(\alpha, \, S_E^\lambda)$:
\begin{equation}\label{Schrodinger_cocycle-conti}
\left(\begin{array}{c}
          x \\
         \xi
        \end{array}
\right)'=(B(E)+F(E,\omega t)) \left(\begin{array}{c}
          x \\
          \xi
        \end{array}
\right),
\end{equation}
Through the change of variables
$$\left(\begin{array}{c}
          x \\
          \xi
        \end{array}\right)=M\left(\begin{array}{c}
          \tilde x \\
          \tilde\xi
        \end{array}\right),$$
$(\omega, \, B(E)+F(E,\cdot))$ is conjugated to
$$\left(\begin{array}{c}
          \tilde x \\
         \tilde\xi
        \end{array}
\right)'=\left(\left(\begin{array}{cc}
            0 & -\nu \\
             \nu & 0
           \end{array}\right)+MF(E,\omega t)M^{-1}\right)\left(\begin{array}{c}
          \tilde x \\
         \tilde\xi
        \end{array}
\right).$$
Then by Theorem \ref{thm_Schro} and \ref{thm_Schro_sobolev},  Theorem \ref{thm_AMO} is shown with
$$\left(\begin{array}{cc}
            b(E,\cdot) & c(E,\cdot) \\
            -a(E,\cdot) & -b(E,\cdot)
           \end{array} \right)=MF(E,\cdot)M^{-1}.$$

Finally we point out that $\rho_{(\omega, \, B(E)+F(E,\cdot))}=\tilde\rho_{(\alpha, \, S_E^\lambda(\cdot))}$, since  $(\alpha, \, S_E^\lambda)$ is the Poincar\'e map of linear system
$(\omega, \, B(E)+F(E,\cdot))$.  Let
$$\tilde\Lambda_{(-p, k)}:=\left\{E\in\overline{\CI}:  \rho_{(\omega, \, B(E)+F(E,\cdot))} =   \frac{ k\alpha-p}{2} =    \min_{j\in \Z}  \left| \frac{ k\alpha}{2} -j\right| \right\} ,$$
then by well-known result of Avila-Jitomirskaya  \cite{AvilaJito},  ${\rm Leb}(\tilde\Lambda_{(-p, k)})>0$, for every $k\in\Z$ such that  $\tilde\Lambda_k$ intersect with ${\CI}$.

\appendix

\section{Proof of Theorem \ref{localemb-sl}}\label{app_proof}

The main ideas of the proof will follow Theorem 3.2 of \cite{YZ2013}, we sketch the proof and point out the differences. First we need the following key observations.

\begin{Lemma} \label{reso}
 For any  $k \in \Z^{d-1},$ and for any $\nu\in C^2(\CI)$  satisfying (\ref{varition}),  there exists $\tilde{k}=\tilde{k}(k) \in \Z$ which is independent of $E$, such that
$$
|\langle k,\mu \rangle+2\nu+\tilde{k}| \in \left[0,\frac{5}{6}\right], \qquad \forall \ E\in \CI .
$$
\end{Lemma}

\proof
 For any given $E$,  we can  define
$\tilde{k}=\tilde{k}(k,E) \in \Z$ by
\begin{equation}\label{sel11}|\langle k,\mu \rangle+2\nu(E)+\tilde{k}|=\inf_{j\in
\Z}|\langle k,\mu \rangle+2\nu(E)+j|,\end{equation}
we only need to show that $\tilde{k}$ can be chosen independent of $E$.

To do this, we only need to consider two extreme cases.  If there exists $E_0\in \CI$ such that $\inf_{k\in \Z}|\langle
k,\mu \rangle+2\nu(E_0)+k|=0$, then $\tilde{k}(k)$ is uniquely defined, and by assumption (\ref{varition}),
$$
|\langle k,\mu \rangle+2\nu(E)+\tilde{k}| \leq 2|\nu(E)-\nu(E_0)| \leq 2 \sup_{E\in\CI}|\nu'(E)|\cdot|\CI| < \frac{1}{3}.
$$
If there exists $E_0\in \CI$ such that $\inf_{k\in \Z}|\langle
k,\mu \rangle+2\nu(E_0)+k|=\frac{1}{2}$, then $\tilde{k}(k)$ is not uniquely defined, and one can
 choose $\tilde{k}(k)$
to be the smaller one which satisfies (\ref{sel11}). By assumption (\ref{varition}), one has
\begin{eqnarray*}
 |\langle k,\mu \rangle+2\nu(E)+\tilde{k}|
&\leq&  |\langle k,\mu \rangle+2\nu(E_0)+\tilde{k}|+  2|\nu(E)-\nu(E_0)|\\
&\leq& \frac{1}{2} +  2 \sup_{E\in\CI}|\nu'(E)|\cdot|\CI| \\
&<& \frac{5}{6}.
\end{eqnarray*}\qed

Once we have Lemma \ref{reso}, we can define the resonance sites $\mathcal{S} \subset \Z^{d}$ as
follows
$$
\mathcal {S}:=\left\{(\tilde{k},k):k\in \Z^{d-1}\right\}.
$$
For any  $f(E,\theta_1,\tilde{\theta})=\sum_{k\in\Z^{d-1}
}\widehat{f}_{\tilde{k},k}(E) e^{2\pi {\rm i}(
\tilde{k}\theta_1+\la k, \tilde{\theta}\ra)}\in
C_h^\omega(\CI\times \T^{d},\C)$, we define its weighted norm by
$$\|f\|_{\nu,h}^\mu:=\sum_{k\in \Z^{d-1}}|\widehat
f_{\tilde{k},k}(E)|_*e^{2\pi |k|(1+|\mu|)h},$$ and then we can
define the linear sub-space $\mathcal{B}_{\nu,h}^\mu( \CI \times \T^{d},\C)$
of
 $C_h^\omega(\CI \times \T^{d},\C)$
$$\mathcal{B}_{\nu,h}^\mu(\CI \times \T^{d},\C):=\left\{f: f(E,\theta_1,\tilde{\theta})=\sum_{k\in\Z^{d-1} }\widehat{f}_{\tilde{k},k}(E)e^{2\pi {\rm i}(
\tilde{k}\theta_1+\la k,
\tilde{\theta}\ra)}, \ \|f\|_{\nu,h}^\mu<\infty  \right\}.$$


  In the following, we will show that  $\mathcal {B}_{\nu,\frac{h}{1+|\mu|}}^\mu(\CI \times \T^{d},\C)$ is actually isomorphic to
$C_h^\omega(\CI \times \T^{d-1},\C)$, therefore   a Banach space. The
space will be used to construct the embedded linear system.

\begin{Proposition}\label{tech}
For any $\nu\in C^2(\CI)$ satisfying (\ref{varition}), the linear operator
\begin{eqnarray*}
 T: \mathcal{B}_{\nu,\frac{h}{1+|\mu|}}^\mu
(\CI \times \T^{d},\C)&\to&C_h^\omega(\CI \times \T^{d-1},\C) \\
 f(E,\theta) &\mapsto& \int_0^1 f(E, t,\tilde\theta+t\mu)e^{4\pi {\rm i}\nu(E) t} \, dt
\end{eqnarray*}
 is  bounded. Moreover,  there exists    numerical constant $c>0$ such that
 $$T^{-1}:C_h^\omega(\CI \times \T^{d-1},\C)\rightarrow \mathcal{B}_{\nu,\frac{h}{1+|\mu|}}^\mu(\CI \times \T^{d},\C) $$
is also  bounded with estimate $\|T^{-1}\| \leq c|\nu|_{*}$.
\end{Proposition}

Before giving the proof of Proposition \ref{tech}, we introduce the following auxiliary function, which is quite important for the proof.
\begin{Lemma}\label{lux}
For the function
\begin{eqnarray*}
H(x)=\left\{
\begin{array}{ccc}     \frac{e^{2\pi {\rm i}x}-1}{2\pi {\rm i} x }, & x \neq  0    \\[1mm]
1, & x=0
\end{array} \right. ,
\end{eqnarray*}
we have $H, \frac{1}{H} \in C^{\infty}[-\frac{5}{6},\frac{5}{6}] $ and
\begin{equation}\label{eh} |H(x)| \in \left[\frac{3}{5\pi},1\right] \qquad \forall \  |x|\leq \frac{5}{6} .\end{equation}
\end{Lemma}

\begin{proof}
By Taylor expansions, one can easily check that  $H\in C^{\infty}[-\frac{5}{6},\frac{5}{6}]$. Since
$$ |H(x)|= \left|\frac{ \sin(\pi x)}{ \pi x}\right|,$$
then (\ref{eh}) follows from the simple fact
$$ \frac{2}{\pi}|t| < |\sin (t)| \leq |t| , \qquad  t\in\left[0,\frac{\pi}{2}\right].$$
Consequently,  $H^{-1}$ is also a $C^{\infty}$ function.
\end{proof}

\noindent
{\bf Proof of Proposition \ref{tech}.} For any $f \in
\mathcal{B}_{\nu,\frac{h}{1+|\mu|}}^\mu (\CI \times \T^{d},\C)$, direct computations show that
 $$T f(E,\theta)= \sum_{k\in \Z^{d-1}} \sum_{(\tilde{k},k
)\in \mathcal {S}}\widehat f_{\tilde k,k}(E)H( \langle k,\mu \rangle+2\nu(E)+\tilde{k} ) e^{ 2\pi {\rm i}
\la k,\theta\ra}.$$
Here we shall use the crucial fact that $\tilde{k}$ is independent of $E$, thus
by Lemma \ref{reso},
$\langle k,\mu \rangle+2\nu(E)+\tilde{k}\in C^2(\CI)$,  and  $|\langle k,\mu \rangle+2\nu(E)+\tilde{k}| \leq \frac{5}{6}$. By  Lemma \ref{lux},
$H( \langle k,\mu \rangle+2\nu(E)+\tilde{k} )$ is well defined and  $H( \langle k,\mu \rangle+2\nu(E)+\tilde{k} ) \in  C^2(\CI)$.
Consequently, there exists  numerical constant $c$ such that
$$\|T f\|_h= \sum_{k\in
\Z^{d-1}} |(\widehat{Tf})_{k}|_* e^{2\pi |k|h} \leq c |\nu|_{*}
\|f\|_{\nu,\frac{h}{1+|\mu|}}^\mu.$$
Hence $T$ is a bounded linear operator.

On the other hand, for any $\varphi \in C_h^\omega(\CI \times \T^{d-1},\C)$, we write  $$\varphi(E,\theta)=\sum_{k\in \Z^{d-1}} \widehat{\varphi }_k(E) e^{2\pi{\rm i} \la k, \theta\ra}.$$  Define
$\widehat{f}_{k_1,k}(E)$ by
\begin{eqnarray*}
\widehat{f}_{k_1,k}(E)=\left\{ \begin{array}{ccc}   \frac{\hat{\varphi}_k(E)}{ H (\langle k,\mu \rangle+2\nu(E)+\tilde{k})},  & k_1=\tilde{k}\\[1mm]
0, & k_1\neq\tilde{k}
\end{array} \right.
\end{eqnarray*}
where  $(\tilde{k},k )\in \mathcal {S}.$ Then one can check that
$$f(E,\theta_1,\tilde{\theta})=\sum_{k\in\Z^{d-1} }\widehat{f}_{k_1,k}(E)e^{2\pi {\rm i} (k_1\theta_1+\la k, \tilde{\theta}\ra)}\in \mathcal
{B}_{\nu,h}^\mu(\CI \times \T^{d},\C)$$ is uniquely defined and it satisfies
$T f(E,0,\tilde{\theta})=\varphi(E,\tilde{\theta}).$
By Lemma \ref{reso} and   Lemma  \ref{lux},  $H^{-1}( \langle k,\mu \rangle+2\nu(E)+\tilde{k} ) \in  C^2(\CI)$,
 consequently, we have
$$\|f\|_{\nu,\frac{h}{1+|\mu|}}^\mu\le c|\nu|_{*}\|\varphi\|_h.$$
Hence $\|T^{-1}\| \leq c|\nu|_{*}$.

 For any $\nu\in C^2(\CI)$, we then can define the Banach
space
\begin{eqnarray*}
\overline{\mathcal {B}}=\left\{\left(\begin{array}{ccc}
{\rm i} f &  g\\
\bar{g} & -{\rm i} f
 \end{array}\right):
 f\in \mathcal{B}_{0,\frac{h}{1+|\mu|}}^\mu(\CI \times \T^{d}, \R), g\in \mathcal{B}_{-\nu,\frac{h}{1+|\mu|}}^\mu( \CI\times \T^d, \C)\right\},
\end{eqnarray*}
then $\overline{\mathcal {B}}\subset C^\omega_{\frac{h}{1+|\mu|}}(\CI \times \T^d, {\rm su}(1,1))$. Note the algebra ${\rm su}(1,1)$ and ${\rm sl}(2,\R)$ are isomorphic with isomorphism given by $B\rightarrow \bar{M}^{-1} B  \bar M$ where
$$\bar M=\left(
\begin{array}{ccc}
 1 &  -{\rm i}\\
 1 &  {\rm i}
 \end{array}\right).$$
 Therefore, we have
$
\mathcal{B}:= \bar M^{-1}\overline{\mathcal {B}} \bar M  \subset C^\omega_{\frac{h}{1+|\mu|}}( \CI \times \T^d, {\rm sl}(2,\R)).
$\qed

 As a corollary of Proposition
\ref{tech}, we have the following:

\begin{Corollary}\label{inverse}
For any $\nu\in C^2(\CI)$ satisfying (\ref{varition}),  then the linear operator
\begin{eqnarray*}
L: \mathcal{B} &\rightarrow&C_h^\omega(\CI \times \T^{d-1},{\rm sl}(2,\R))  \\
  F&\mapsto& \int_0^1 e^{-2\pi \nu J s}F(s,\theta+s\mu)e^{ 2\pi \nu J s} \, ds
\end{eqnarray*}
is bounded. Moreover,  there exists    numerical constant $c>0$ such that
$$L^{-1}:C_h^\omega(\CI \times \T^{d-1},{\rm sl}(2,\R))\rightarrow
 \mathcal{B}$$
is  bounded with  $\|L^{-1}\|\leq  c|\nu|_{*}$.
\end{Corollary}
\begin{proof}  It is an immediate corollary of corollary of Proposition
\ref{tech},  similar proof can be found in Corollary 3.1 of \cite{YZ2013}. We omit the details. \end{proof}

\noindent
\textbf{Proof of Theorem \ref{localemb-sl}.}
Now we can finish the whole proof of Theorem \ref{localemb-sl}.   We will use quantitative Implicit Function Theorem (c.f. Theorem 3.1 of \cite{YZ2013}) to prove the result.  Suppose that $ \Phi^t(E,\theta)$ is the fundamental
solution matrix of $(\ref{al-ref1})$,
$$
\Phi^t(E,\theta)=e^{2\pi \nu(E) J t}\left({\rm Id} +\int_0^t e^{-2\pi \nu(E) J
s}F(E,\theta+s\omega) \Phi^s(E,\theta)ds \right),$$ where ${\rm Id}$
denotes the identity matrix.

 We will show that the cocycle
$(\mu, e^{2\pi\nu J} e^{G(E,\tilde{\theta})})$ can be embedded into the
linear system $(\ref{al-ref1}),$ which means
$\Phi^1(E, 0,\tilde{\theta})=e^{2\pi\nu(E) J} e^{G(E,\tilde{\theta})}$,
i.e.,
$$
e^{2\pi\nu J}\left({\rm Id} +\int_0^1 e^{-2\pi\nu J  s}F(E,s,\tilde{\theta}+s\mu)  \Phi^s(E,0,\tilde{\theta})ds \right)=e^{2\pi\nu J} e^{G(E,\tilde{\theta})}.$$
We then construct the nonlinear functional
\begin{eqnarray*}
\Psi:\mathcal{B}\times C_h^\omega(\CI \times \T^{d-1},{\rm sl}(2,\R)) \to C_h^\omega(\CI \times \T^{d-1},{\rm gl}(2,\R))
\end{eqnarray*}
by defining
$$\Psi (F,G):={\rm Id} +\int_0^1 e^{-2\pi\nu J s}F(E,s,\tilde{\theta}+s\mu)
\Phi^s(E,0,\tilde{\theta})ds - e^{G(E,\tilde{\theta})}$$
Immediate check shows that $\Psi(0,0)=0,$  $\|\Psi(0,G)\|\leq
\|G\|_h$, and

\begin{eqnarray*}D_F\Psi(F,G)(\widetilde{F})&=&\int_0^1 e^{- 2\pi\nu J
s}\widetilde{F}(E,s,\tilde{\theta}+s\omega)\Phi^s(E,0,\tilde{\theta})ds\\
& & \ \  \   +\int_0^1 e^{- 2\pi\nu J s}F(E,s,\tilde{\theta}+s\mu) D_F
\Phi^s(0,\tilde{\theta})\widetilde{F}(E,s,\tilde{\theta}+s\mu)ds.\end{eqnarray*}
Consequently, we have
$$D_F\Psi(0,0)(\widetilde{F})=\int_0^1 e^{-2\pi\nu J s}\widetilde{F}(E,s,\tilde{\theta}+s\mu)e^{2\pi\nu J s} ds.$$
By Corollary \ref{inverse}, $D_F\Psi(0,0)^{-1}:
C_h^\omega(\CI \times \T^{d-1},{\rm sl}(2,\R)) \rightarrow \mathcal {B}$ is a bounded
linear operator with estimate $\|D_F\Psi(0,0)^{-1}\|\leq c|\nu|_*$.

The rest proof are quite standard, one can consult Theorem 3.2 of \cite{YZ2013} for details, we omit the details.\qed

\

\noindent {\bf Acknowledgements}

The authors would like to thank Prof. D. Bambusi and Prof. J. You for their interests and fruitful discussions which helped to improve the proof.
They also appreciate the anonymous referees for helpful remarks and suggestions in modifying this manuscript.
Z. Zhao would like to thank the support of Visiting Scholars of Shanghai Jiaotong University (SJTU) and Visiting Scholars of Chern Institute of Mathematics (CIM) during his visits.





\

\end{document}